\documentclass[Thereside,11pt]{article}
\usepackage{tikz}
\usepackage [latin1]{inputenc}
\usepackage[english]{babel}
\usepackage{amsmath,amscd}
\usepackage{amsthm}
\usepackage{algorithm}
\usepackage{cite}
\usepackage{marvosym}
\usepackage{amsfonts}
\usepackage{amssymb}
\usepackage{mathrsfs}
\usepackage{graphicx}
\usepackage{colortbl,dcolumn}
\usepackage{marvosym}
\usepackage{ifsym}
\usepackage{paralist}
\usepackage{xcolor}
\usepackage{array,color}
\usepackage[colorlinks,
citecolor=blue,
linkcolor=blue,
backref=page
allcolors=black]{hyperref} 
\graphicspath{{figures/}}
\allowdisplaybreaks
       \newtheorem{lemma}{\bf Lemma}[section]
       \newtheorem{theorem}{\bf Theorem}[section]
       
       \newtheorem{corollary}{\bf Corollary}[section]
       
       \newtheorem{remark}{\bf Remark}[section]

       \numberwithin{equation}{section}

\begin{document}
\title{{\sl Quantitative uniform exponential acceleration of  averages along decaying waves}}
\author{Zhicheng Tong $^{\mathcal{z}}$, Yong Li $^{\mathcal{x}}$}

\renewcommand{\thefootnote}{}
\footnotetext{\hspace*{-6mm}

\begin{tabular}{l   l}	$^\mathcal{z}$~School of Mathematics, Jilin University, Changchun 130012, P. R.  China. \url{tongzc20@mails.jlu.edu.cn}\\	$^{\mathcal{x}}$~The corresponding author. School of Mathematics, Jilin University, Changchun 130012, P. R.  China; \\  Center for Mathematics and Interdisciplinary Sciences,  Northeast  Normal  University, Changchun 130024, \\P. R. China.  \url{liyong@jlu.edu.cn}
\end{tabular}}

\date{}
\maketitle

\begin{abstract}
In this study, utilizing a specific exponential weighting function, we investigate the uniform exponential convergence of weighted Birkhoff averages along decaying waves and delve into several related variants. A key distinction from traditional scenarios is evident here: despite reduced regularity in observables, our method still maintains exponential convergence. In particular, we develop new techniques that yield very precise rates of exponential convergence, as evidenced by numerical simulations. Furthermore, this innovative approach extends to quantitative analyses involving different weighting functions employed by others, surpassing the limitations inherent in prior research. It also enhances the exponential convergence rates of weighted Birkhoff averages along quasi-periodic orbits via analytic observables. To the best of our knowledge, this is the first result on the uniform exponential acceleration beyond averages along  quasi-periodic or almost periodic orbits,  particularly from a quantitative perspective.
\end{abstract}
\tableofcontents

\section{Introduction}

As one of the foundational theories in dynamical systems, ergodic theory first originated from Birkhoff and von Neumann.  This seminal work has given rise to an array of diverse formulations and has been extensively applied across various domains. It is a well-established fact, as noted by Krengel \cite{MR0510630}, that the convergence rates of Birkhoff averages in ergodic theory are generally slow, and can even be \textit{arbitrarily slow} for certain counterexamples.	Similar assertions were reaffirmed  by Ryzhikov \cite{MR4602432}, while recent results on the convergence rate within the Birkhoff ergodic theorem can be found in Podvigin's comprehensive survey \cite{Podviginsurvey}. We also mention  the previous valuable survey \cite{MR3643963} by Kachurovski\u{\i} and Podvigin. Furthermore, Yoccoz's counterexamples \cite{MR604672,MR1367354}, constructed via rotations that are extremely Liouvillean--nearly rational--on the finite-dimensional torus, deserve mentioning.		 Such slow convergence is \textit{ubiquitous} in ergodic theory and is \textit{intrinsically inescapable}, and it would be at most  $\mathcal{O}(N^{-1})$ in non-trivial cases with time length $ N $, i.e., the observables are  non-constant, as detailed by Kachurovski\u{\i} \cite{MR1422228}. More frustratingly, the pursuit of high-precision numerical results may demand computational durations on the order of billions of years, a point elaborated by Das and Yorke in \cite{MR3755876}.

Given the acknowledged slow convergence in ergodic theory, the weighting method is crucial for accelerating computations in both mathematics and mechanics. There is a growing interest in identifying suitable weighting functions to enhance the convergence rates of ergodic averages. To investigate quasi-periodic perturbations of quasi-periodic flows, Laskar employed a weighting function $ \sin^2(\pi x) $ to accelerate computational processes, as detailed in \cite{MR1234445,MR1222935,MR1720890}. Furthermore, Laskar claimed that a specific exponential weighting function possesses superior asymptotic characteristics, although he did not implement it or provide evidence of its convergence properties, as noted in \cite{MR1720890}. It is worth highlighting that the resulting convergence rate can be demonstrated to exceed that of any arbitrary polynomial, which we shall detail later. To be more precise, Laskar employed  the following weighting function to investigate  ergodicity within dynamical systems:
\begin{equation}\label{DWFUN}
{w}\left( x \right):= \left\{ \begin{array}{ll}
	{\left( {\int_0^1 {\exp \left( { - {s^{ - 1}}{{\left( {1 - s} \right)}^{ - 1}}} \right)ds} } \right)^{ - 1}}\exp \left( { - {x^{ - 1}}{{\left( {1 - x} \right)}^{ - 1}}} \right),&x \in \left( {0,1} \right),  \hfill \\
	0,&x \notin \left(0,1\right), \hfill \\
\end{array}  \right.
\end{equation}
and the corresponding weighted Birkhoff average of an sufficiently smooth observable $ f $ evaluated along a trajectory of length $ N $ reads 
\begin{equation}\label{DWWBA}
	{\rm{WB}}_{N}\left( f \right)\left( x \right):=\frac{1}{{{A_N}}}\sum\limits_{n = 0}^{N - 1} {w\left( {n/N} \right)f\left( {{T^n}x} \right)} ,
\end{equation}
provided $ {A_N}: = \sum\nolimits_{n = 0}^{N - 1} {w\left( {n/N} \right)}  $, and $ T $ is a quasi-periodic mapping on the $ d $-torus $ \mathbb{T}^d $ with $ d \in \mathbb{N}^+ $, $ x\in \mathbb{T}^d $ is an initial point.  It is evident  that $ w \in C_0^\infty \left( {\left[ {0,1} \right]} \right) $ and the integral  $ \int_0^1 {w\left( x \right)dx} =1 $. From a statistical standpoint, this approach effectively reduces the impact of the initial and terminal data, thereby highlighting those in the central range. This emphasis aligns with the concept of averaging, consequently leading to an intuitively rapid uniform convergence to the spatial average $ \int_{{\mathbb{T}^d}} {f\left( \theta  \right)d\theta }  $. 	However, there are more profound reasons behind this phenomenon, which have been  illustrated by the authors in \cite{TLCesaro}.

Back to the main topic, following Laskar's significant numerical discovery, Das and Yorke \cite{MR3755876} provided rigorous proofs for  quasi-periodic systems with Diophantine rotations. To be more precise, they demonstrated that under $ 	C^\infty $ observables,  the weighted Birkhoff averages can achieve rapid convergence of arbitrary polynomial order. However, more rigorous theoretical work is still lacking in this area. To this end, inspired by the fundamental work of Das and Yorke, the authors recently made interesting findings in \cite{TLCesaro,TLmultiple,MR4768308}--namely that universally--or saying in terms of  full  measure,  the weighted Birkhoff averages (as well as the  Ces\`{a}ro weighted averages)  exhibit exponential convergence as long as the observables are analytic and irrespective of finite dimensions (quasi-periodic) or infinite dimensions (almost periodic).  More generally, the authors also discussed abstract scenarios admitting weaker nonresonance and regularity  and  present certain balancing criteria to achieve rapid convergence. Due to brevity considerations,  these details are not elaborated upon here.

Based on these theoretical foundations, many researchers have utilized this weighting method to perform numerical calculations on various physical systems which has significantly increased their efficiency. On this aspect, see Calleja et al. \cite{MR4665059}, Das et al. \cite{MR3718733}, Sander and Meiss \cite{MR4104977,MR4845476,arXiv:2409.08496}, Meiss and Sander \cite{MR4322369}, Duignan and Meiss \cite{MR4582163}, Blessing and Mireles James \cite{MR4771043} among several others; these contributions provide strong evidence supporting rapid convergence.  Specifically, Duignan and Meiss \cite{MR4582163} employed this weighting method to distinguish between regular (in fact, quasi-periodic there) and chaotic systems. To be more precise, under the weighted Birkhoff averages via \eqref{DWFUN}, regular systems typically demonstrate a rapid exponential convergence. In contrast, chaotic systems exhibit at most  ${{{\mathcal{O}(N^{-1})}}} $ slow convergence. This raises an intriguing question: \textit{what other scenarios might these so-called regular systems applicable to weighted Birkhoff averages encompass?} 

\setcounter{footnote}{0}
\renewcommand{\thefootnote}{\fnsymbol{footnote}}
The primary motivations behind this work are twofold: \textbf{(M1)} We aim to delve deeper into the practical applications of weighted Birkhoff averages beyond quasi-periodic and almost periodic systems which admit zero Lyapunov exponents. 	We believe at least that, for dynamical systems having \textit{negative maximal Lyapunov exponents},  averages along trajectories, once weighted by the function \eqref{DWFUN}, should also exhibit \textit{exponential convergence} under certain generic conditions. \textbf{(M2)} Currently, the analysis of the convergence rates for the weighted Birkhoff averages with the weighting function \eqref{DWFUN} is almost entirely \textit{qualitative}\footnote{{{For more detailed explanations, please refer to Section \ref{DWSECREVISITED}.}}} (at least, not particularly precise when compared with numerical simulations of the convergence rates). We suspect that this imprecision is due to limitations in previous techniques. Therefore, we aim to develop new approaches to provide a \textit{quantitative} version of these convergence rates. 

\setcounter{footnote}{0}
\renewcommand{\thefootnote}{\fnsymbol{footnote}}
Based on the above considerations, we first start with some typical examples, namely exponentially decaying waves (e.g., $e^{-\lambda x}\sin (\rho x)$ and $e^{-\lambda x}\cos (\rho x)$ with certain $\lambda>0$ and $\rho \in \mathbb{R}$), which constitute some fundamental solutions among others for linear matrix differential equations described as $x'=Ax$ (other solutions have more polynomial parts than these, depending on the multiplicity of the eigenvalues, see Section \ref{DWSEC5}), where all eigenvalues of the constant matrix $A$  lie inside the unit circle $ \mathbb{S}^1 $. Such decaying waves can also be interpreted as projections onto planes from spiral waves within $3$-dimensional space. For further insights into the  connections  between decaying waves and dynamics systems,  we refer readers to  Liu et al. \cite{MR3797619,MR4603828}, Wang et al. \cite{MR4393340}, Xing et al. \cite{MR4562787} inclusive of their references which provide an exploration into this domain\footnote{{{There, the so-called affine periodic solutions are investigated, satisfying
			\[
			f(s+T, t) = Q f\left(s, Q^{-1} t\right), \;\; \text{where } f: \mathbb{R}^1 \times \mathbb{R}^d \to \mathbb{R}^d \text{ is continuous}, \;\; Q \in \mathrm{GL}(d), \;\; \forall (s, t)\in \mathbb{R}^1 \times \mathbb{R}^d.
			\]
			This represents a more general formulation of decaying waves.}}}.  As a main contribution, we demonstrate that Laskar's weighted averaging method not only significantly accelerates the Birkhoff averages along quasi-periodic or almost periodic orbits, but also proves effective for other scenarios, involving {{weighted averages}} of decaying waves, with extremely rapid exponential convergence rates contrasting with the universal one-order polynomial type observed in unweighted cases (Theorem \ref{SWT1} and Remark \ref{DWRE12}).  Specifically, our \textit{quantitative} convergence rates prove to be \textit{remarkably precise}, as evidenced by the numerical simulations illustrated in Figure \ref{FIGDW1} within Section \ref{DWSEC7}. Furthermore, we proceed to discuss more general decaying waves in series forms and arrive at \textit{optimal} results somewhat unexpectedly (Theorem \ref{DWCORO2}).

Before delving  into the main subject, let us first introduce some basic notations for the sake of simplicity. We define that in the limit process as $ x \to +\infty  $, the statements $ f_1(x)=\mathcal{O}\left(f_2(x)\right)  $ and $ f_1(x)=\mathcal{O}^{\#}\left(f_2(x)\right)  $  imply  there exists a positive constant $ c >0$, such that  $ |f_1(x)| \leqslant c f_2(x) $ and $  c^{-1} f_2(x) \leqslant |f_1(x)| \leqslant c f_2(x) $, respectively.
We also denote by $ D^nf $  the $ n $-th derivative of a sufficiently smooth function $ f $. Additionally, for any fixed number $ x\in \mathbb{R} $, we define $ {\left\| x \right\|_\mathbb{T}} :={\inf _{n \in \mathbb{Z}}}\left| {x - 2\pi n} \right|$ as the distance to the closest element of the set $ 2\pi \mathbb{Z} $, or equivalently, it represents the  metric on the standard torus $ \mathbb{T}: = \mathbb{R}/2\pi \mathbb{Z} $ identified with $ [0,2\pi) $.

In what follows, we proceed to present our main results.  Consider the unweighted time average of the  decaying wave as
\[{{\rm{DW}}_N}\left( {\lambda ,\rho,\theta  } \right): = \frac{1}{N}\sum\limits_{n = 0}^{N - 1} {{e^{ - \lambda n}}\sin \left( {\theta  + n\rho } \right)}, \]
as well as the $ w $-weighted type given by
\[{\rm{WDW}}_N\left( {\lambda ,\rho ,\theta } \right): = \frac{1}{{{A_N}}}\sum\limits_{n = 0}^{N - 1} {w\left( {n/N} \right){e^{ - \lambda n}}\sin \left( {\theta  + n\rho } \right)}.\]

The first  main result in this paper is stated as follows.
\begin{theorem}\label{SWT1}
(I) For any fixed $ \lambda  > 0 $ and $ \rho  \in \mathbb{R} $, 
\[\left| {{\rm{WDW}}_N\left( {\lambda ,\rho ,\theta } \right)} \right| = \mathcal{O}\left( {\exp \left( { - {\xi }\sqrt N } \right)} \right),\;\;N \to  + \infty \]
uniformly holds with $ \theta  \in \mathbb{R} $, where $ \xi=\xi(\lambda,\rho) = \sqrt {2\lambda }  + {e^{ - 1}}\sqrt {{\lambda ^2} + \left\| \rho  \right\|_{\mathbb{T}}^2}  >0 $.  In particular, the control coefficient is independent of $ \rho $, and is uniformly bounded for large $ \lambda $.
\\
(II) However, for any fixed $ \lambda  > 0 $ and $ \rho,\theta  \in \mathbb{R} $,  there exists infinitely many $ N \in \mathbb{N}^+ $ such that 
\[\left| {{{\rm{DW}}_N}\left( {\lambda ,\rho,\theta  } \right)} \right| = {\mathcal{O}^\# }\left( {{N^{ - 1}}} \right),\;\;N \to  + \infty,\]
whenever $ {e^{ - \lambda }}\sin \left( {\rho  - \theta } \right) + \sin \theta  \ne 0 $.
\end{theorem}
\begin{remark}\label{DWRE12}
For fixed $ \lambda>0 $ and $ \rho \in \mathbb{R} $, the set of such $ \theta \in \mathbb{R} $ has at most   zero Lebesgue measure.  Therefore, the $ \mathcal{O}^{{{\#}}}\left(N^{-1}\right) $ convergence for the unweighted decaying waves is universal, or commonly referred to as almost certain.  This is completely consistent with the zero-one law in \cite{Podviginsurvey}. 
\end{remark}

Theorem \ref{SWT1} provides an upper bound on the convergence rate of the {{weighted average}} $ {{\rm{WDW}}_N\left( {\lambda ,\rho ,\theta } \right)} $, which depends simultaneously on the decaying parameter $ \lambda $ and the rotating parameter $ \rho $. As we will demonstrate in the numerical simulation within Section \ref{DWSEC7}, this estimate is remarkably precise. 

The following conclusion concerning the averages of superimposed waves is a direct corollary of Theorem \ref{SWT1}.
\begin{corollary}\label{DWCORO11}
For $ k \in \mathbb{N} $, define $ {g_k}\left( {x,\theta } \right): = {e^{ - {\lambda _k}x}}\sin \left( {\theta  + x{\rho _k}} \right) $, where $ \mathop {\inf }\nolimits_{k \in \mathbb{N}} {\lambda _k} > 0 $ and $ {\rho _k},\theta  \in \mathbb{R} $. Then for any sequence $ {\left\{ {{c_k}} \right\}_{k \in \mathbb{N}}} \in {l^1} $, the time average $ \frac{1}{A_N}\sum\nolimits_{n = 0}^{N - 1} {w(n/N)g\left( {n,\theta } \right)}  $ of the superimposed wave $ g\left( {x,\theta } \right) = \sum\nolimits_{k \in \mathbb{N}} {{c_k}{g_k}\left( {x,\theta } \right)}  $   converges uniformly (with respect to $ \theta $) and exponentially to $ 0 $.
\end{corollary}
	Indeed, when the decaying parameters $ \lambda_k =\lambda>0$, Corollary \ref{DWCORO11} covers the cases of decaying quasi-periodic or decaying almost periodic. For a given set of rotation parameters $ \{\rho_k\}_{k=1}^{d} $ with $ 1 \leqslant d  <+\infty  $ ($ =+\infty $), if for any $ 0 \ne k \in \mathbb{Z}^d $, $ \sum\nolimits_{j=1}^{d} {k_j  \rho_j} \ne 0 $, then the motion exhibited  is quasi-periodic (almost periodic). See \cite{MR4768308} for further characterization of these two concepts, as well as specific examples in Kozlov \cite{MR4233648}, for instance. However, the control coefficient in Theorem \ref{SWT1} may tend to $ +\infty $ as the decaying parameter $ \lambda $ tends to $ 0^+ $. Therefore, if the case considered in Corollary \ref{DWCORO11} allows the decaying parameters $ \lambda_k $ to tend to $ 0^+ $ (acting similarly to Bessel functions), then with the variation in the arithmetic properties of the rotation parameters $ \rho_k $, more complicated small denominator problems may arise. This needs to be dealt with in combination with the techniques in \cite{MR4768308,TLCesaro,TLmultiple} (see also Section \ref{DWSECREVISITED}), which we prefer not to discuss here.

To enhance the applicability of our results, we present the quantitative Theorem \ref{DWCORO2} below. Prior to this, %we provide some fundamental concepts. %see Fan and Jiang \cite{MR1860763} for instance. Recall that a right continuous and increasing function $ \varpi $ with $ \varpi(0)=0 $ is called a modulus of continuity of a function $ f $ on $ \mathbb{T} $, if  
we denote by $ A(\mathbb{T}) $ the space consisting of continuous functions on $ \mathbb{T} $ having an absolutely convergent Fourier series, i.e., the functions $ f\left( x \right) = \sum\nolimits_{k \in \mathbb{Z}} {{{\hat f}_k}{e^{ikx}}}  $ for which $ \sum\nolimits_{k \in \mathbb{Z}} {|{{\hat f}_k}|}  <  + \infty  $. With the norm defined by $ {\left\| f \right\|_{A\left( \mathbb{T} \right)}}: = \sum\nolimits_{k \in \mathbb{Z}} {|{{\hat f}_k}|}  $, $ {A\left( \mathbb{T} \right)} $ is a Banach space  isometric to $ l^1 $ with algebra property, see Katznelson \cite{MR2039503} for instance. For $0< \alpha< 1 $, we denote by $ C^{\alpha}(\mathbb{T}) $ the space of $ \alpha $-H\"older functions endowed with the norm
\[{\left\| f \right\|_{{C^\alpha }\left( \mathbb{T} \right)}}: = \mathop {\sup }\limits_{x \in \mathbb{T}} \left| {f\left( x \right)} \right| + \mathop {\sup }\limits_{x,y \in \mathbb{T},\;x \ne y} \frac{{\left| {f\left( x \right) - f\left( y \right)} \right|}}{{{{\left| {x - y} \right|}^\alpha }}} <  + \infty .\]
 We also denote by  $ C^{n}(\mathbb{T}) $ the space of functions on $ \mathbb{T} $ with $ n $-order continuous derivatives, where $ n \in \mathbb{N} $.

\begin{theorem}\label{DWCORO2}
	Let $ \lambda>0 $, $ \rho,\theta \in \mathbb{R} $ be given.
\begin{itemize}
\item[(I)] For any observable $ \mathscr{P}(x) \in A(\mathbb{T}) $, it holds uniformly with respect to $ \theta $  that 
\begin{equation}\label{DWCORO2-1}
	\left| {\frac{1}{{{A_N}}}\sum\limits_{n = 0}^{N - 1} {w\left( {n/N} \right){e^{ - \lambda n}}\mathscr{P}\left( {\theta  + n\rho } \right)}  } \right| = \mathcal{O}\left( {\exp \left( { - \xi' \sqrt N } \right)} \right),\;\;N \to  + \infty,
\end{equation}
where  $ \xi'=\xi(\lambda,0) = \sqrt {2\lambda }  + {e^{ - 1}}\lambda >0 $. A sufficient case is $  {\mathscr{P}}(x) \in C^{\alpha}(\mathbb{T}) $ with $ \alpha>1/2 $. In particular, for any trigonometric polynomial of order $ \ell $ on $ \mathbb{T} $ without the Fourier constant term, i.e., $ \mathscr{P}\left( x \right) = \sum\nolimits_{n = 1}^\ell  {\left( {{a_n}\cos nx + {b_n}\sin nx} \right)}   $, the index $ \xi' $ of the exponential rate in \eqref{DWCORO2-1} could be improved to $ {\xi ^ * } = \sqrt {2\lambda }  + {e^{ - 1}}\sqrt {{\lambda ^2} + \mathop {\min }\nolimits_{1 \leqslant j \leqslant \ell } \left\| {j\rho } \right\|_\mathbb{T}^2}  $.

\item[(II)] For any analytic observable $ \mathscr{Q}(x) $ on $[-1,1]$, it holds uniformly with respect to $ \theta $  that 
\begin{equation}\label{DWCORO2-2}
	\left| {\frac{1}{{{A_N}}}\sum\limits_{n = 0}^{N - 1} {w\left( {n/N} \right)\mathscr{Q}\left( {{e^{ - \lambda n}}\sin \left( {\theta  + n\rho } \right)} \right)}  - \mathscr{Q}\left( 0 \right)} \right| = \mathcal{O}\left( {\exp \left( { - \xi' \sqrt N } \right)} \right),\;\;N \to  + \infty,
\end{equation}
where  $ \xi'=\xi(\lambda,0) = \sqrt {2\lambda }  + {e^{ - 1}}\lambda >0 $.  In particular, for any  polynomial of order $ \nu $, i.e., $ \mathscr{Q}\left( x \right) = \sum\nolimits_{j = 0}^\nu  {{\mathscr{Q}_j}{x^j}}  $, the index $ \xi' $ of the exponential rate in \eqref{DWCORO2-2} could be improved to $ {\xi _ * } = \min \{ {\xi _ * ^{(1)},\xi _ * ^{(2)}} \}$, provided
\[\xi _ * ^{(1)} = \min \left\{ {\xi \left( {j\lambda ,0} \right) = \sqrt {2j\lambda }  + {e^{ - 1}}j\lambda :\; \text{for even $ j $ with $ 0 \leqslant j \leqslant \nu $ and $ {\mathscr{Q}_j} \ne 0 $}} \right\}\]
and
\[\xi _ * ^{(2)} = \min \left\{ {\xi \left( {j\lambda ,j\rho } \right) = \sqrt {2j\lambda }  + {e^{ - 1}}\sqrt {{j^2}{\lambda ^2} + \left\| {j\rho } \right\|_\mathbb{T}^2} :\; \text{for odd $ j $ with $ 0 \leqslant j \leqslant \nu $ and $ {\mathscr{Q}_j} \ne 0 $}} \right\}.\]
\item[(III)] For any observable $ \mathscr{K}(x) $ on $[-1,1]$ with $ \left| {\mathscr{K}\left( x \right) - \mathscr{K}\left( 0 \right)} \right| \leqslant M{x^{2\tau }} $ for some $ M>0 $ and $ \tau \in \mathbb{N}^+ $, it holds uniformly with respect to $ \theta $  that 
\[\left| {\frac{1}{{{A_N}}}\sum\limits_{n = 0}^{N - 1} {w\left( {n/N} \right)\mathscr{K}\left( {{e^{ - \lambda n}}\sin \left( {\theta  + n\rho } \right)} \right)}  - \mathscr{K}\left( 0 \right)} \right| = \mathcal{O}\left( {\exp \left( { - \xi_*^* \sqrt N } \right)} \right),\;\;N \to  + \infty,\]
where $ {\xi _ *^* } = \xi \left( {2\tau \lambda ,0} \right) = 2\sqrt {\tau \lambda }  + 2{e^{ - 1}}\tau \lambda  > 0 $. A sufficient case is $ \mathscr{K}(x) \in {C^{2\tau }}\left( \mathbb{T} \right) $ satisfying $ {D^j}\mathscr{K}\left( 0 \right) = 0 $ for $ 1 \leqslant j \leqslant 2\tau  - 1 $ and $ {D^{2\tau }}\mathscr{K}\left( 0 \right) \ne 0 $.
\end{itemize}
\end{theorem}
\begin{remark}
The sufficient case $  {\mathscr{P}}(x) \in C^{\alpha}(\mathbb{T}) $ with $ \alpha>1/2 $ is indeed optimal for $ \mathscr{P}(x) \in A(\mathbb{T}) $ in (I), namely there exists a counterexample (Hardy-Littlewood series) $ \tilde f(x) \in C^{1/2}(\mathbb{T})$ such that $ \tilde f (x)\notin A(\mathbb{T}) $, see  Zygmund \cite{MR0107776} and  Katznelson \cite{MR2039503} for instance.
\end{remark}
\begin{remark}
 For trigonometric series, the absolute summability in (I) is almost optimal in order to guarantee the uniform convergence with respect to $ \theta $ in \eqref{DWCORO2-1}. For example, considering $ a_n>0 $ and $ \sum\nolimits_{n \in \mathbb{Z}} {{a_n}}  =  + \infty  $, we construct the counterexample as $ \bar f\left( x \right) = \sum\nolimits_{n \in \mathbb{Z}} {{a_n}\cos nx}  $. Then $ \bar f\left( 0 \right) =  + \infty  $. Taking $ \rho=0 $ and $ \theta>0  $ sufficiently close to $ 0 $, the weighed decaying wave becomes $ \left( {\frac{1}{{{A_N}}}\sum\nolimits_{n = 0}^{N - 1} {w\left( {n/N} \right){e^{ - \lambda n}}} } \right)\bar f\left( \theta  \right) $, which exhibits non-uniform convergence for such  $ \theta   $.
\end{remark}
 Recall that in \cite{MR4322369,MR3755876,MR4768308}, achieving fast convergence of the weighted Birkhoff average necessitates strong regularity for the observable, e.g., the $ C^\infty $ regularity or even analyticity. However, as shown  for the decaying wave in (I), incorporating the decaying part compensates for the lack of regularity, thereby naturally achieving the  exponential convergence. Moreover, the observable in (III) could admit  very weak regularity, for example, discontinuity except of $ x=0 $.

 {{Recall the motivations (M1) and (M2) discussed previously. It is important to note that our primary focus lies in understanding the impact of weights on averages from the perspective of dynamical systems, rather than from the standpoint of numerical simulations. Historically, studying the convergence of weighted averages has generally been challenging, and obtaining quantitative results is even more difficult. However, our results do have certain practical applications. For instance, when the observed quantity can be linearly decomposed into a quasi-periodic (or almost periodic) term and a decaying wave term, such weighted averages can effectively accelerate computations exponentially. It is necessary to point out that the results we have established regarding decaying waves are not applicable to the weighted computation of linear-exponentially decaying waves alone from a practical standpoint. Specifically, when there is no rotational component, taking the last term alone can yield a faster exponential expression, thus eliminating the need for weighted averaging.}}

The remainder of this paper is organized as follows. To highlight the ideas underlying our proof, we provide a detailed analysis of Theorem \ref{SWT1} in Section \ref{DWSEC2}. Building upon this foundation, we give a more concise proof for Theorem \ref{DWCORO2} in Section \ref{DWSEC3}. We also \textit{quantitatively} discuss  more general weighting functions in Section \ref{DWSEC4} and demonstrate their potential to lead exponential convergence of weighted Birkhoff averages, thereby providing \textit{theoretical validation} for the numerical results by Das et al. \cite{MR3718733}, Duignan and Meiss \cite{MR4582163}, and Calleja et al. \cite{MR4665059}, among others. Although Theorems \ref{SWT1} and \ref{DWCORO2} consider discrete cases, we illustrate in Section \ref{DWSEC5} that such results extend naturally to continuous cases as well, and indeed, even simpler. In addition to the exponentially decaying waves addressed in Theorems \ref{SWT1} and \ref{DWCORO2}, our newly developed techniques are also applicable to the {{weighted averages}} associated with more general types of decaying waves, especially orbits from  certain nonlinear dynamical systems having \textit{negative maximum Lyapunov exponents}, as detailed in  Section \ref{DWSEC5}.  In Section \ref{DWSECREVISITED}, following careful approaches introduced in this paper, we revisit the exponential convergence rates of weighted Birkhoff averages along quasi-periodic orbits as initially discussed in \cite{MR4768308}, contributing an \textit{enhanced quantitative} result (Theorem \ref{DWJQ}) to the existing research. As part of the main contributions, we show that: 
\begin{itemize}
	\item (Part of Theorem \ref{DWJQ})
In the weighted Birkhoff average \eqref{DWWBA}, for almost all quasi-periodic mappings $ T $, whenever the observable is analytic, it converges to the spatial average $ \int_{{\mathbb{T}^d}} {f\left( \theta  \right)d\theta }  $ at an exponential rate of $  \mathcal{O}\left( {\exp \left( { - {c_{\rm I}}{N^{\frac{1}{{d + 2}}}}{{\left( {\ln N} \right)}^{ - \frac{\zeta}{{d + 2}}}}} \right)} \right) $ in the $ C^0 $ topology, where $ \zeta > 1 $ is arbitrarily given, and ${c_{\rm I}} >0 $ is some universal constant. 
\end{itemize}
We also establish a result of nearly linear-exponential convergence  in Theorem \ref{DWRYZS}. Finally,  we present in Section \ref{DWSEC7}  numerical simulation for a \textit{universal} example which showcase the precision of our results, from a \textit{quantitative} perspective.

\section{Proof of Theorem \ref{SWT1}}\label{DWSEC2}
\subsection{Proof of (I): Exponential convergence of the weighted type}
We first prove (I). With the Poisson summation formula (see  \cite{MR3243734,MR1970295} for instance), we have
\begin{align}
{{\rm{WDW}}_N}\left( {\lambda ,\rho ,\theta } \right): &= \frac{1}{{{A_N}}}\sum\limits_{n = 0}^{N - 1} {w\left( {n/N} \right){e^{ - \lambda n}}\sin \left( {\theta  + n\rho } \right)}\notag \\ 
	&  = \operatorname{Im} \left\{ {\frac{1}{{{A_N}}}\sum\limits_{n = 0}^{N - 1} {w\left( {n/N} \right){e^{ - \lambda n + i\left( {\theta  + n\rho } \right)}}} } \right\} \notag \\
	& = \operatorname{Im} \left\{ {\frac{1}{{{A_N}}}\sum\limits_{n =  - \infty }^\infty  {\int_{ - \infty }^{ + \infty } {w\left( {t/N} \right){e^{ - \lambda t + i\left( {\theta  + t\rho } \right)}}{e^{ - 2\pi int}}dt} } } \right\} \notag \\
	&  = \operatorname{Im} \left\{ {\frac{N{{{e^{i\theta }}}}}{{{A_N}}}\sum\limits_{n =  - \infty }^\infty  {\int_0^1 {w\left( y \right){e^{N\left( { - \lambda  + i\rho  - 2\pi in} \right)z}}d{{z}}}    } } \right\},\notag
\end{align}
and this leads to 
\begin{equation}\label{SWZJ1}
	\left| {{\rm{WDW}}_N}\left( {\lambda ,\rho ,\theta } \right)\right| \leqslant \frac{N}{{{A_N}}}\sum\limits_{n =  - \infty }^\infty  {\left| {\int_0^1 {w\left( y \right){e^{N\left( { - \lambda  + i\rho  - 2\pi in} \right)y}}dy} } \right|} .
\end{equation}
This eliminates the influence  of the initial phase parameter  $ \theta $.  By integrating by parts, it is evident to get
 \[\int_0^1 {w\left( y \right){e^{N\left( { - \lambda  + i\rho  - 2\pi in} \right)y}}dy}  = \frac{{\int_0^1 {\left({D^m}w\left( y \right)\right){e^{N\left( { - \lambda  + i\rho  - 2\pi in} \right)y}}dy} }}{{{{\left( {N\left( { - \lambda  + i\rho  - 2\pi in} \right)} \right)}^m}}}.\]
 Note that $ \left| {{e^{N\left( { - \lambda  + i\rho  - 2\pi in} \right)y}}} \right|  = {e^{ - \lambda Ny}} $ and $ \left| {N\left( { - \lambda  + i\rho  - 2\pi in} \right)} \right| = N\sqrt {{\lambda ^2} + {{\left( {\rho  - 2\pi n} \right)}^2}}  $. Then it follows that
\begin{equation}\label{SWZJ2}
	\left| {\int_0^1 {w\left( y \right){e^{N\left( { - \lambda  + i\rho  - 2\pi in} \right)y}}dy} } \right| \leqslant \frac{{{{\left\| {\left| {{D^m}w\left( y \right)} \right|{e^{ - \lambda Ny}}} \right\|}_{{L^1(0,1)}}}}}{{{N^m}{{\left( {{\lambda ^2} + {{\left( {\rho  - 2\pi n} \right)}^2}} \right)}^{m/2}}}}.
\end{equation}
 Next, we need certain  accurate asymptotic estimates for  $ {\left\| {\left| {{D^m}w\left( y \right)} \right|{e^{ - \lambda Ny}}} \right\|_{{L^1}}} $, as demonstrated in the following Lemma \ref{SWLEMMA1}.

\begin{lemma}\label{SWLEMMA1}
There exist absolute constants $ C_1,\lambda_2>1 $,   such that for all $ m,N \in \mathbb{N}^+ $,
\begin{equation}\label{DWLEMMA21}
	{\left\| {\left| {{D^m}w\left( y \right)} \right|{e^{ - \lambda Ny}}} \right\|_{{L^1(0,1)}}} \leqslant {C_1}\lambda _2^m{m^m}\exp \left( { - \sqrt {2\lambda N} } \right){N^{\frac{{m - 1}}{2}}}.
\end{equation}
\end{lemma}
\begin{proof}
Note that $ w(z) $ is locally holomorphic. Fix $ x \in (0,1) $, let us choose some $ \lambda_1  \in \left( {0,1} \right) $ {{(independent of $ x $)}} such that for $ \delta  = \lambda_1 \min \left\{ {x,1 - x} \right\} $, it holds that 
\[\mathop {\sup }\limits_{s \in \partial B\left( {x,\delta } \right)} \left| {{w}\left( s \right)} \right| \leqslant \max \left\{ {\exp \left( { - \frac{1}{2}{x^{ - 1}}} \right),\exp \left( { - \frac{1}{2}{{\left( {1 - x} \right)}^{ - 1}}} \right)} \right\},\]
where $ \partial B\left( {x,\delta } \right) = \left\{ {\zeta  \in \mathbb{C}:\left| {\zeta  - x} \right| = \delta } \right\} $.
Then, using Cauchy's integral formula and the triangle inequality, we have that for $ x \in (0,1) $,
\begin{align}
	\left| {{D^m}{w}\left( x \right)} \right| &= \left| {\frac{{m!}}{{2\pi i}}\int_{\partial B\left( {x,\delta } \right)} {\frac{{{w}\left( \zeta  \right)}}{{{{\left( {\zeta  - x} \right)}^{m + 1}}}}d\zeta } } \right| \notag \\
	& \leqslant \frac{{m!}}{{2\pi }}\mathop {\sup }\limits_{s \in \partial B\left( {x,\delta } \right)} \left| {{w}\left( s \right)} \right|\int_{\partial B\left( {x,\delta } \right)} {\frac{1}{{{{\left| {\zeta  - x} \right|}^{m + 1}}}}d\zeta } \notag \\
	& \leqslant \frac{{m!}}{{2\pi }}\frac{{2\pi \delta }}{{{\delta ^{m + 1}}}}\max \left\{ {\exp \left( { - \frac{1}{2}{x^{ - 1}}} \right),\exp \left( { - \frac{1}{2}{{\left( {1 - x} \right)}^{ - 1}}} \right)} \right\}\notag \\
	\label{SWPQ1}& \leqslant \frac{{m!}}{{{\lambda_1 ^m}}}\max \left\{ {{x^{ - m}}\exp \left( { - \frac{1}{2}{x^{ - 1}}} \right),{{\left( {1 - x} \right)}^{ - m}}\exp \left( { - \frac{1}{2}{{\left( {1 - x} \right)}^{ -1}}} \right)} \right\}.
\end{align} 
We mention that  \eqref{SWPQ1} serves as an  extension of Problem 4* in Chapter 5 of \cite{MR1970295}. Following \eqref{SWPQ1}, one can also prove that $ w $ is in a Denjoy-Carleman class (see the definition in \cite{MR4447133}). With \eqref{SWPQ1}, we derive that
\begin{align}
&{\left\| {\left| {{D^m}w\left( y \right)} \right|{e^{ - \lambda Ny}}} \right\|_{{L^1(0,1)}}}  \notag \\
  \leqslant &\frac{{m!}}{{\lambda _1^m}}\int_0^1 {\max \left\{ {{y^{ - m}}\exp \left( { - \frac{1}{2}{y^{ - 1}}} \right),{{\left( {1 - y} \right)}^{ - m}}\exp \left( { - \frac{1}{2}{{\left( {1 - y} \right)}^{ - 1}}} \right)} \right\}{e^{ - \lambda Ny}}dy}  \notag \\
  \leqslant& \frac{{2m!}}{{\lambda _1^m}}\int_0^{\frac{1}{2}} {{y^{ - m}}\exp \left( { - \frac{1}{2}{y^{ - 1}} - \lambda Ny} \right)dy}  \notag \\
 = &\frac{{2m!}}{{\lambda _1^m}}\int_2^{ + \infty } {{u^{m - 2}}\exp \left( { - \frac{1}{2}u - \lambda N{u^{ - 1}}} \right)du}  \notag \\
 \label{SWJIFEN1}\leqslant& \frac{{2m!}}{{\lambda _1^m}}\int_0^{ + \infty } {{u^{m - \frac{1}{2}}}\exp \left( { - \frac{1}{2}u - \lambda N{u^{ - 1}}} \right)du} .
\end{align}
The crucial point is to estimate the asymptotic behavior of the integral in \eqref{SWJIFEN1}. For $ A,B>0 $, consider the parameterized integral $ \Phi \left( {A,B} \right) $ defined by
\[\Phi \left( {A,B} \right): = \int_0^{ + \infty } {\exp \left( { - {{\left( {As - B{s^{ - 1}}} \right)}^2}} \right)ds}. \]
Utilizing the Cauchy-Schl\"omilch transformation, one can verify that $ \Phi \left( {A,B} \right) $ is independent of $ B $, and it is indeed equal to $ \frac{1}{A}\int_0^{ + \infty } {\exp \left( { - {q^2}} \right)dq}  = \frac{{\sqrt \pi  }}{{2A}} $.
Then for any given $ 0<\sigma<1 $ and $\eta>0 $, it holds 
\[\Psi \left( {\sigma ,\eta } \right): = \int_0^{ + \infty } {\exp \left( { - \sigma {s^2} - \eta {s^{ - 2}}} \right)ds}  = \exp \left( { - 2\sqrt {\sigma \eta } } \right)\frac{{\sqrt \pi  }}{{2\sqrt \sigma  }}.\]
On the one hand, 
\begin{align}
\left| {\partial _\sigma ^{m - 1}\Psi \left( {\sigma ,\eta } \right)} \right| &= \int_0^{ + \infty } {{s^{2\left( {m - 1} \right)}}\exp \left( { - \sigma {s^2} - \eta {s^{ - 2}}} \right)ds} \notag \\
\label{SWDENGJIA1}& = \frac{1}{2}\int_0^{ + \infty } {{u^{m - 1 - \frac{1}{2}}}\exp \left( { - \sigma u - \eta {u^{ - 1}}} \right)du} .
\end{align}
On the other hand,
\begin{equation}\label{SWDENGJIA2}
	\left| {\partial _\sigma ^{m - 1}\Psi \left( {\sigma ,\eta } \right)} \right| = \left| {\partial _\sigma ^{m - 1}\left( {\exp \left( { - 2\sqrt {\sigma \eta } } \right)\frac{{\sqrt \pi  }}{{2\sqrt \sigma  }}} \right)} \right|.
\end{equation}
One notices that the right hand side of \eqref{SWDENGJIA2} contains at most $ 2^m $ terms, and the power terms of $ \sigma $ can be dominated by $ {\sigma ^{ - m + 1/2}} $ due to $ 0<\sigma<1 $, and  those derivatives with respect to $ {\exp \left( { - 2\sqrt {\sigma \eta } } \right)} $ will eventually generate $ \exp \left( { - 2\sqrt {\sigma \eta } } \right){\eta ^{\frac{{m - 1}}{2}}} $. On these grounds, there exists some absolute constant $ C_1>0 $ (independent of $ m,\eta $) such that {{for all ${\eta\gg 1}$ and $\sigma$ bounded away below from $ 0 $,}}
\begin{equation}\label{SWdaoguji}
	\left| {\partial _\sigma ^{m - 1}\Psi \left( {\sigma ,\eta } \right)} \right| \leqslant {C_1}{2^{2m}}\exp \left( { - 2\sqrt {\sigma \eta } } \right){\eta ^{\frac{{m - 1}}{2}}}.
\end{equation}
Combining \eqref{SWJIFEN1}, \eqref{SWDENGJIA1}, \eqref{SWDENGJIA2},  \eqref{SWdaoguji} (with $ \sigma  = \frac{1}{2} $ and $ \eta  = \lambda N $) and utilizing  Stirling's approximation $ m!\sim\sqrt {2\pi m} {m^m}{e^{ - m}} $ as $ m\to +\infty $,  we obtain the desired estimate with some $ \lambda_2>1 $:
\[{\left\| {\left| {{D^m}w\left( y \right)} \right|{e^{ - \lambda Ny}}} \right\|_{{L^1(0,1)}}} \leqslant {C_1}\lambda _2^m{m^m}\exp \left( { - \sqrt {2\lambda N} } \right){N^{\frac{{m - 1}}{2}}},\;\;\forall m,N \in {\mathbb{N}^ + }.\]
\end{proof}

It is evident that there exists some absolute constant $ C_2>0 $ such that $ N/{A_N} \leqslant {C_2} $. Therefore, with \eqref{SWZJ1}, \eqref{SWZJ2} and Lemma \ref{SWLEMMA1}, we have that
\begin{equation}\label{SWZong1}
	\left| {{{\rm{WDW}}_N}\left( {\lambda ,\rho ,\theta } \right)} \right| \leqslant {C_3}\exp \left( { - \sqrt {2\lambda N} } \right) \cdot {N^{ - \frac{1}{2}}} \sum\limits_{n =  - \infty }^\infty  {{\mathcal{Q}_n}\left( {\lambda ,\rho ,m,N} \right)} 
\end{equation}
holds for $ C_3:=C_1C_2>0 $ and 
\[{\mathcal{Q}_n}\left( {\lambda ,\rho ,m,N} \right)=\mathcal{Q}_n: = \frac{{\lambda _2^m{m^m}}}{{{N^{m/2}}{{\left( {{\lambda ^2} + {{\left( {\rho  - 2\pi n} \right)}^2}} \right)}^{m/2}}}}.\]
In the next lemma, we will demonstrate that $ {N^{ - \frac{1}{2}}}\sum\nolimits_{n =  - \infty }^\infty  {{\mathcal{Q}_n}}  $  exhibits   exponential decay with respect to $ N $. The crucial insight in the proof is that by adjusting {{the number of integrations by parts $ m $}}  (with respect to the parameters $ \lambda ,\rho ,n,N $, etc.), we can  minimize $ \mathcal{Q}_n $ within the permissible range, leading to its eventual exponential decay. Furthermore, we will employ a truncation method to demonstrate that the summation over $ n $ does not affect  the overall exponential decay.

\begin{lemma}\label{SWLEMMA2}
There exists some absolute constant $ C_8>0 $, such that
\[{N^{ - \frac{1}{2}}}\sum\limits_{n =  - \infty }^\infty  {{\mathcal{Q}_n}\left( {\lambda ,\rho ,m,N} \right)}  \leqslant {C_8}\exp \left( { - {e^{ - 1}}\sqrt {{\lambda ^2} + \left\| \rho  \right\|_{\mathbb{T}}^2} \sqrt N } \right)\]
holds for  $ N $ sufficiently large.
\end{lemma}
\begin{proof}
For the sake of brevity, define $ {\mathcal{Q}_n} = \exp \left( {\Xi \left( m \right)} \right) $, where 
\[\Xi \left( x \right): = x\left( {\ln x + \ln {\lambda _3} - \frac{1}{2}\ln N} \right),\;\;{\lambda _3}: = {\lambda _2}{\left( {{\lambda ^2} + {{\left( {\rho  - 2\pi n} \right)}^2}} \right)^{ - 1/2}}.\]
According to monotonicity analysis, we have $ {\mathop {\min }\nolimits_{x > 0}}\Xi \left( x \right) = \Xi \left( {\lambda _3^{ - 1}{e^{ - 1}}\sqrt N } \right) $. Then by setting $ m \sim \lambda _3^{ - 1}{e^{ - 1}}\sqrt N  $ as $ N \to +\infty $, we arrive at
\begin{equation}\label{SWQN}
	{{\mathcal{Q}_n}} \leqslant {C_4}\exp \left( { - {e^{ - 1}}\lambda _2^{ - 1}\sqrt {{\lambda ^2} + {{\left( {\rho  - 2\pi n} \right)}^2}} \sqrt N } \right)
\end{equation}
for some absolute constant $ C_4>0 $ (independent of $ n,N $).  To estimate $ {N^{ - \frac{1}{2}}}\sum\nolimits_{n =  - \infty }^\infty  {{\mathcal{Q}_n}}  $, we consider partitioning $ \mathbb{Z} $ into a union of two disjoint sets, namely $ \mathbb{Z} = {\Lambda _1} \cup {\Lambda _2} $ with
\[{\Lambda _1}: = \left\{ {n \in \mathbb{Z}:\left| {\rho  - 2\pi n} \right| \leqslant \sqrt N } \right\},\;\;{\Lambda _2}: = \left\{ {n \in \mathbb{Z}:\left| {\rho  - 2\pi n} \right| > \sqrt N } \right\}.\]
Now, by utilizing \eqref{SWQN}, we obtain that
\begin{align}
\sum\limits_{n =  - \infty }^\infty  {{\mathcal{Q}_n}\left( {\lambda ,\rho ,m,N} \right)}  &\leqslant {C_4}\left( {\sum\limits_{n \in {\Lambda _1}} { + \sum\limits_{n \in {\Lambda _2}} {\exp \left( { - {e^{ - 1}}\lambda _2^{ - 1}\sqrt {{\lambda ^2} + {{\left( {\rho  - 2\pi n} \right)}^2}} \sqrt N } \right)} } } \right)\notag \\
\label{SWQNQIUHE}: &= {C_4}\left( {{\mathcal{S}_{{\Lambda _1}}} + {\mathcal{S}_{{\Lambda _2}}}} \right). 
\end{align}
For convenience, we declare that the following absolute positive constants $ C_5 $ through $ C_8 $ are independent of $ N $. On the one hand,
\begin{align}
	{\mathcal{S}_{{\Lambda _1}}}& \leqslant \# {\Lambda _1} \cdot \exp \left( { - {e^{ - 1}}\lambda _2^{ - 1}\sqrt {{\lambda ^2}{\text{ + }}\left\| \rho  \right\|_{\mathbb{T}}^2} \sqrt N } \right) \notag \\
\label{SWSGAMA1}	&\leqslant {C_5}\sqrt N \exp \left( { - {e^{ - 1}}\lambda _2^{ - 1}\sqrt {{\lambda ^2}{\text{ + }}\left\| \rho  \right\|_{\mathbb{T}}^2} \sqrt N } \right).
\end{align}
On the other hand, 
\begin{align}
{\mathcal{S}_{{\Lambda _2}}} &\leqslant {C_6}\int_{\sqrt N }^{ + \infty } {\exp \left( { - {e^{ - 1}}\lambda _2^{ - 1}\sqrt {{\lambda ^2} + {\tau ^2}} \sqrt N } \right)d\tau } \notag \\
& \leqslant {C_7}\exp \left( { - {e^{ - 1}}\lambda _2^{ - 1}\sqrt {{\lambda ^2} + N} \sqrt N } \right)\notag \\
\label{SWSGAMA2}& = \mathcal{O}\left( {\exp \left( { - N} \right)} \right).
\end{align}
Now, substituting \eqref{SWSGAMA1} and \eqref{SWSGAMA2} into \eqref{SWQNQIUHE} yields 
\begin{align*}
{N^{ - \frac{1}{2}}}\sum\limits_{n =  - \infty }^\infty  {{Q_n}\left( {\lambda ,\rho ,m,N} \right)} & \leqslant {C_4}{N^{ - \frac{1}{2}}}\left( {{C_5}\sqrt N \exp \left( { - {e^{ - 1}}\lambda _2^{ - 1}\sqrt {{\lambda ^2} + \left\| \rho  \right\|_{\mathbb{T}}^2}  \sqrt N } \right) + \mathcal{O}\left( {\exp \left( { - N} \right)} \right)} \right)\\
& \leqslant {C_8}\exp \left( { - {e^{ - 1}}\sqrt {{\lambda ^2} + \left\| \rho  \right\|_{\mathbb{T}}^2} \sqrt N } \right)
\end{align*}
due to $ \lambda_2>1 $, which completes the proof.
\end{proof}

Applying Lemma \ref{SWLEMMA2} to \eqref{SWZong1}, we finally prove that
\begin{align}
\left| {{{\rm{WDW}}_N}\left( {\lambda ,\rho ,\theta } \right)} \right| &= \mathcal{O}\left( {\exp \left( { - \sqrt {2\lambda N} } \right)} \right) \cdot \mathcal{O}\left( {\exp \left( { - {e^{ - 1}}\sqrt {{\lambda ^2} + \left\| \rho  \right\|_{\mathbb{T}}^2} \sqrt N } \right)} \right)\notag \\
&= \mathcal{O}\left( {\exp \left( { - {\xi}\sqrt N } \right)} \right),\notag
\end{align}
provided with $ \xi=\xi(\lambda,\rho) = \sqrt {2\lambda }  + {e^{ - 1}}\sqrt {{\lambda ^2} + \left\| \rho  \right\|_{\mathbb{T}}^2}  >0 $. Moreover, it is evident from the previous  analysis that the control coefficient could be independent of $ \rho $, and indeed, it could be uniformly bounded for large $ \lambda $. This proves (I) of Theorem \ref{SWT1}.

\subsection{Proof of (II): $ \mathcal{O}\left(N^{-1}\right) $  convergence of the unweighted type}
As to (II), it is evident to calculate that
\begin{align*}
{\rm{DW}}_N\left( {\lambda ,\theta ,\rho } \right): &= \frac{1}{N}\sum\limits_{n = 0}^{N - 1} {{e^{ - \lambda n}}\sin \left( {\theta  + n\rho } \right)}  = \operatorname{Im} \left\{ {\frac{1}{N}{e^{i\theta }}\sum\limits_{n = 0}^{N - 1} {{e^{\left( { - \lambda  + i\rho } \right)n}}} } \right\}\\
& = \operatorname{Im} \left\{ {\frac{1}{N}{e^{i\theta }}\frac{{1 - {e^{\left( { - \lambda  + i\rho } \right)N}}}}{{1 - {e^{ - \lambda  + i\rho }}}}} \right\} = \frac{{{e^{ - \lambda }}\sin \left( {\rho  - \theta } \right) + \sin \theta }}{{{{\left( {1 - {e^{ - \lambda }}\cos \rho } \right)}^2} + {{\left( {{e^{ - \lambda }}\sin \rho } \right)}^2}}} \cdot \frac{1}{N} + \mathcal{O}\big( {{e^{ - \lambda N}}} \big).
\end{align*}
This leads to $ \mathcal{O}\left(N^{-1}\right) $  convergence of $ {\rm{DW}}_N\left( {\lambda ,\theta ,\rho } \right) $ whenever $ {e^{ - \lambda }}\sin \left( {\rho  - \theta } \right) + \sin \theta  \ne 0 $, which proves (II) of Theorem \ref{SWT1}.

\subsection{Proof of Corollary \ref{DWCORO11}}
Corollary \ref{DWCORO11} is a direct consequence of Theorem \ref{SWT1}, for which we provide a brief proof here. According to part (I) of Theorem \ref{SWT1}, for each decaying wave $ {g_k}\left( {x,\theta } \right) $, it converges uniformly to $ 0 $ at least at an exponential  rate $ \mathcal{O}(\exp(-(\sqrt {2{{\inf }_k}{\lambda _k}}  + {e^{ - 1}}{\inf _k}{\lambda _k})\sqrt N )) $ after weighted averaging due to the uniform lower bound on $ \lambda_k $, and the control coefficient is independent of $  k $. Consequently, the absolute summability of the sequence $ \{c_k\}_{k \in \mathbb{N}} $ allows for the interchange of order, thus completing the proof.

\section{Proof of Theorem \ref{DWCORO2}}\label{DWSEC3}
We emphasize that the dependence of the control coefficient in Theorem \ref{SWT1} on the parameters is essential to prove Theorem \ref{DWCORO2}.\vspace{2mm}
\\
  Proof of (I): We only prove the latter two conclusions, as the first conclusion remains the same.  In view of  $  {\mathscr{P}}(x) \in C^{\alpha}(\mathbb{T}) $ with $ \alpha>1/2 $, we write $ {\mathscr{P}}\left( x \right) \sim \sum\nolimits_{k \in \mathbb{Z}} {{{\hat {\mathscr{P}}}_k}{e^{ikx}}}  $. In what follows, we prove that the Fourier series is indeed absolutely summable, which is known as the Bernstein Theorem. Below we provide a brief proof for the completeness. For any $ m \in \mathbb{N}^+ $, define $ {{\mathscr{P}}_h}\left( x \right): = {\mathscr{P}}\left( {x - h} \right) $ for $ h := 2\pi /\left( {3 \cdot {2^m}} \right) $. Then it can be verified that
  \[\mathop {\sup }\limits_{{2^m} \leqslant k < {2^{m + 1}}} | {{e^{ - ikh}} - 1} | \geqslant \mathop {\sup }\limits_{t \in [2\pi /3,4\pi /3)} | {{e^{ - it}} - 1} | = \mathop {\sup }\limits_{t \in [2\pi /3,4\pi /3)} \sqrt {2 - 2\cos t}  = \sqrt 3 .\]
  This leads to 
\begin{align}
  	 \sum\limits_{{2^m} \leqslant k < {2^{m + 1}}} {|{{\hat{\mathscr{P}}}_k}{|^2}}  &\leqslant \frac{1}{3}\sum\limits_{{2^m} \leqslant k < {2^{m + 1}}} {{{| {{e^{ - ikh}} - 1} |}^2}|{{\hat {\mathscr{P}}}_k}|^2}  \leqslant \frac{1}{3}\sum\limits_{k \in \mathbb{Z}} {{{| {{e^{ - ikh}} - 1} |}^2}|{{\hat {\mathscr{P}}}_k}|^2} \notag \\
  \label{DWHOLDER}	& = \frac{\pi }{3}\left\| {{{\mathscr{P}}_h} - {\mathscr{P}}} \right\|_{{L^2}\left( \mathbb{T} \right)}^2 \leqslant \frac{\pi }{3}\int_\mathbb{T} {{h^{2\alpha }}\left\| {\mathscr{P}} \right\|_{{C^\alpha }\left( \mathbb{T} \right)}^2dx}  \leqslant \frac{{{C_\alpha }}}{{{2^{2\alpha m}}}}\left\| {\mathscr{P}} \right\|_{{C^\alpha }\left( \mathbb{T} \right)}^2
  \end{align}
with some $ C_{\alpha}>0 $, where the H\"older continuity is used  in \eqref{DWHOLDER}. Then by Cauchy's inequality, we have
\[\sum\limits_{{2^m} \leqslant k < {2^{m + 1}}} {|{{\hat {\mathscr{P}}}_k}|}  \leqslant {\Bigg( {\sum\limits_{{2^m} \leqslant k < {2^{m + 1}}} {|{{\hat {\mathscr{P}}}_k}{|^2}} } \Bigg)^{1/2}}{\Bigg( {\sum\limits_{{2^m} \leqslant k < {2^{m + 1}}} {{1^2}} } \Bigg)^{1/2}} \leqslant C_\alpha ^{1/2}{\left\| {\mathscr{P}} \right\|_{{C^\alpha }\left( \mathbb{T} \right)}} \cdot {2^{\frac{{m+1}}{2} - \alpha m}}.\]
By summing up $ k $ with $ m \in \mathbb{N}^+ $ and utilizing $ \alpha>0 $, we obtain the absolute summable Fourier series of $ {\mathscr{P}} $, which is  converges pointwise  to $ {\mathscr{P}} $ on $ \mathbb{T} $. 

  For an individual $ {e^{ik\left( {\theta  + n\rho } \right)}} $ with $ k \in \mathbb{Z} $, the weighted average of the decaying wave admits exponential convergence in Theorem \ref{SWT1} with the decaying index $ \xi_k  = \xi \left( {\lambda ,k\rho } \right) \geqslant \xi \left( {\lambda ,0} \right)=\sqrt {2\lambda }  + {e^{ - 1}}\lambda: = \xi ' >0$. Therefore, by summing up $ k $ and utilizing $ \sum\nolimits_{ k \in \mathbb{Z}} {| {{{\hat {\mathscr{P}}}_k}} |}  <  + \infty  $, we prove the first claim in (I). One notices that for general $ \mathscr{P} $, we cannot obtain a more precise decaying index under this approach, as in Theorem  \ref{SWT1}, replacing $ \sin $ by $ \cos $ (the proof is completely same) and taking $ \theta=\rho=0 $, we arrive at the decaying index $ \xi' $ here.  Another reason is that as $ |k| $ increases, $ \{\|k \rho\|_\mathbb{T}\} _{k \in \mathbb{Z}}$ has a subsequence that tends to $ 0 $ (whether $ \rho $ is rational or irrational), and therefore it does not have a positive bound from below. However, for any trigonometric polynomial of order $ \ell $ on $ \mathbb{T} $ without the Fourier constant term, the decaying  index $ \xi' $ of the exponential rate could be improved to $ {\xi ^ * } = \sqrt {2\lambda }  + {e^{ - 1}}\sqrt {{\lambda ^2} + \mathop {\min }\nolimits_{1 \leqslant j \leqslant \ell } \left\| {j\rho } \right\|_\mathbb{T}^2}  $, in a similar way. To be more precise, for any trigonometric part,  the corresponding decaying index in Theorem \ref{SWT1} has a uniform lower bound, as $ \xi _k = \xi \left( {\lambda ,k\rho } \right) \geqslant {\xi ^ * } $ for all $ 1 \leqslant k \leqslant \ell  $.  \vspace{2mm}
 \\
 Proof of (II): As to (II), the analyticity of $ \mathscr{Q} $ yields $ \mathscr{Q}\left( x \right) = \sum\nolimits_{k \in \mathbb{N}} {\frac{{{D^k}\mathscr{Q}\left( 0 \right)}}{{k!}}{x^k}}  $ for all $ x \in [-1,1] $, and $ \sum\nolimits_{k = 1}^\infty  {\frac{{\left| {{D^k}\mathscr{Q}\left( 0 \right)} \right|}}{{k!}}}  <  + \infty  $. Note that for $ k \in \mathbb{N}^+ $, we have
\begin{equation}\label{DWSANJIAO}
	{\left( {{e^{ - \lambda n}}\sin \left( {\theta  + n\rho } \right)} \right)^k} = \frac{{{e^{ - \lambda kn}}}}{{{2^k}{i^k}}}\sum\limits_{j = 0}^k {{{\left( { - 1} \right)}^j}C_k^j\frac{{{e^{i\left( {k - 2j} \right)\left( {\theta  + n\rho } \right)}} + {{\left( { - 1} \right)}^k}{e^{ - i\left( {k - 2j} \right)\left( {\theta  + n\rho } \right)}}}}{2}}.
\end{equation}
Then it is evident to verify that the  summation of the coefficients of the trigonometric polynomial does not exceed $ \frac{1}{{{2^k}}}\sum\nolimits_{j = 0}^k {C_k^j}  = 1 $. Therefore, with $ \frac{1}{{{A_N}}}\sum\nolimits_{n = 0}^{N - 1} {w\left( {n/N} \right)}  = 1 $, the strict monotonicity of $ \xi=\xi(\lambda,\rho) $ with respect to $ \lambda $ and $ \|\rho\|_\mathbb{T} $ and Theorem \ref{SWT1}, we deduce the desired conclusion in (II) as follows:
\begin{align*}
\frac{1}{{{A_N}}}\sum\limits_{n = 0}^{N - 1} {w\left( {n/N} \right)\mathscr{Q}\left( {{e^{ - \lambda n}}\sin \left( {\theta  + n\rho } \right)} \right)}  &= \mathscr{Q}\left( 0 \right) + \sum\limits_{k = 1}^\infty  {\frac{{{D^k}\mathscr{Q}\left( 0 \right)}}{{k!}} \cdot \mathcal{O}\left( {\exp \left( { - \xi \left( {k\lambda ,k\rho } \right)\sqrt N } \right)} \right)} \\
&  = \mathscr{Q}\left( 0 \right) + \mathcal{O}\left( {\sum\limits_{k = 1}^\infty  {\frac{{\left| {{D^k}\mathscr{Q}\left( 0 \right)} \right|}}{{k!}}} } \right) \cdot \mathcal{O}\left( {\exp \left( { - \xi \left( {\lambda,0 } \right)\sqrt N } \right)} \right)\\
& = \mathscr{Q}\left( 0 \right) + \mathcal{O}\left( {\exp \left( { - \xi' \sqrt N } \right)} \right),
\end{align*}
where $ \xi'=\xi(\lambda,0) = \sqrt {2\lambda }  + {e^{ - 1}}\lambda >0 $. When considering  polynomials of order $ \nu $, i.e., $ \mathscr{Q}\left( x \right) = \sum\nolimits_{j = 0}^\nu  {{\mathscr{Q}_j}{x^j}}  $, it can be observed from the trigonometric part in  \eqref{DWSANJIAO} that for $ \sin(k\left( {\theta  + n\rho } \right)) $ with $ 0 \leqslant k \leqslant \nu $, the even case and the odd case are indeed different. For the even case, the lowest convergence rate is $ \mathcal{O}( {\exp ( { - \xi _ * ^{(1)}\sqrt N } )} ) $
with
\[\xi _ * ^{(1)} = \min \left\{ {\xi \left( {j\lambda ,0} \right) = \sqrt {2j\lambda }  + {e^{ - 1}}j\lambda :\; \text{for even $ j $ with $ 0 \leqslant j \leqslant \nu $ and $ {\mathscr{Q}_j} \ne 0 $}} \right\},\]
due to the existence of the constant term. While for the odd case where the constant term vanishes, we obtain the lowest convergence rate as $ \mathcal{O}( {\exp ( { - \xi _ * ^{(2)}\sqrt N } )} ) $,
provided with
\[\xi _ * ^{(2)} = \min \left\{ {\xi \left( {j\lambda ,j\rho } \right) = \sqrt {2j\lambda }  + {e^{ - 1}}\sqrt {{j^2}{\lambda ^2} + \left\| {j\rho } \right\|_\mathbb{T}^2} :\; \text{for odd $ j $ with $ 0 \leqslant j \leqslant \nu $ and $ {\mathscr{Q}_j} \ne 0 $}} \right\}.\]
Combining the above two convergence rates we complete the proof of (II).
   \vspace{2mm}
\\
  Proof of (III): The analysis is actually similar to (II), whenever we observe that
  \begin{align*}
  	 & \left| {\frac{1}{{{A_N}}}\sum\limits_{n = 0}^{N - 1} {w\left( {n/N} \right){\mathscr{K}}\left( {{e^{ - \lambda n}}\sin \left( {\theta  + n\rho } \right)} \right)}  - {\mathscr{K}}\left( 0 \right)} \right|\\
  	\leqslant&  \frac{M}{{{A_N}}}\sum\limits_{n = 0}^{N - 1} {w\left( {n/N} \right){e^{ - 2\lambda n}}{{\sin }^{2\tau }}\left( {\theta  + n\rho } \right)} ,
  \end{align*}
  which yields the convergence rate  $ \mathcal{O}( {\exp ( { - \xi _ * ^{*}\sqrt N } )} ) $ with $ {\xi _ *^* } = \xi \left( {2\tau \lambda ,0} \right) = 2\sqrt {\tau \lambda }  + 2{e^{ - 1}}\tau \lambda  > 0 $, as desired.

\section{Further discussions on general exponential weighting functions}\label{DWSEC4}
Historically, there have been various types of weighting functions for weighted Birkhoff  averages. For instance, Laskar \cite{MR1234445,MR1222935,MR1720890} utilized  $\sin^2(\pi x)$, and Das et al. \cite{MR3718733} as well as the authors \cite{TLCesaro}  employed $x(1-x)$  for comparative purposes (in relation to  \eqref{DWFUN}). It is also worth mentioning that very recently, Ruth and Bindel \cite{MR4833591} introduced a co-called ``tuned filter''  that possesses similar desirable properties to \eqref{DWFUN} in applications. However, some weighting functions may not possess the property of  being $C_0^\infty$ smoothness, which is crucial for achieving universal arbitrary polynomial or even exponential convergence. As a result, under such weighting functions, the weighted Birkhoff averages typically only exhibit convergence up to finite polynomial order. We do not intend to discuss them, but instead focus on some  analogues of \eqref{DWFUN} that have been verified in numerical simulations to exhibit significantly superior exponential acceleration effects; see  Das et al. \cite{MR3718733}, Duignan and Meiss  \cite{MR4582163}, and Calleja  et al. \cite{MR4665059}  on this aspect.

\subsection{Exponential weighting function with two regularity parameters $ p,q $}\label{DWPQZHISHU}
For $ p,q> 0 $, consider a general exponential weighting function $ {w}_{p,q}(x) $ on $ \mathbb{R} $ defined by
\[	{w}_{p,q}\left( x \right):= \left\{ \begin{array}{ll}
	{\left( {\int_0^1 {\exp \left( { - {s^{ - p}}{{\left( {1 - s} \right)}^{ - q}}} \right)ds} } \right)^{ - 1}}\exp \left( { - {x^{ - p}}{{\left( {1 - x} \right)}^{ - q}}} \right),&x \in \left( {0,1} \right),  \hfill \\
	0,&x \notin \left(0,1\right). \hfill \\
\end{array}  \right.\]
In particular, $ w(x)=w_{1,1}(x) $. We refer to $ p $ and $ q $ as \textit{regularity parameters}, because they characterize the asymptotic properties of the weighting function $ {w}_{p,q}(x) $ at the boundaries of the compact support.  It also   preserves the elementary properties of $ w(x) $, as $ w_{p,q}(x) \in C_0^\infty \left( {\left[ {0,1} \right]} \right) $ and $ \int_0^1 {w_{p,q}\left( x \right)dx} =1 $. Das et al. \cite{MR3718733} and Calleja  et al. \cite{MR4665059}	 utilized this weighting function with $ p=q=2 $ to find rotation numbers (the latter concerned  the spin-orbit problem with tidal torque), and observed the so-called super-convergence. For general $ p,q>0 $, it is evident to prove the arbitrary polynomial convergence via $  w_{p,q}(x) $ in weighted Birkhoff averages, however, it is \textit{non-trivial} to obtain the exponential convergence. The authors provided in \cite{MR4768308,TLCesaro,TLmultiple} a useful approach to achieve  the exponential convergence by estimating $ \| D^mw(x) \|_{L^1(0,1)} $ for sufficiently large $ m $, following the spirit of induction. Unfortunately, it becomes much more difficult when dealing with the weighting function $ w_{p,q}(x) $ for $ p,q\ne 1 $ from induction,   as integrating by parts multiple times will lead to more complicated parts. More importantly, according to this method, it is almost impossible to obtain  such precise convergence rate estimates in our paper (Theorems \ref{SWT1} and \ref{DWCORO2}). Consequently, there \textit{does not} exist any theoretical result that guarantees the exponential convergence of the weighted Birkhoff averages in \cite{MR3718733,MR4665059}.  It is therefore natural that ones should consider the following questions:

\begin{itemize}
\item [(Q1)] Do the weighed Birkhoff averages in \cite{MR4768308,TLCesaro,TLmultiple,MR3718733,MR4665059} via $ w_{p,q}(x) $ still admit the exponential convergence?

\item [(Q2)] How about the convergence rates of the weighted decaying waves via $ w_{p,q}(x) $?
\end{itemize}

The answer for question (Q1), is \textit{affirmative}. We first establish a quantitative lemma for $ w_{p,q}(x) $, as an extension of Lemma 5.3 in \cite{MR4768308}  concerning $ w(x)=w_{1,1}(x) $.  This approach is entirely distinct from the ones in \cite{MR3755876, MR3718733,MR4768308}, yet it proves to be highly effective. Subsequently, the theoretical exponential convergence discussed  in \cite{MR4768308,TLCesaro,TLmultiple} via $ w_{p,q}(x) $ can be obtained directly, by replacing $ \beta $ in Lemma 5.3 in \cite{MR4768308} with any number greater than $ 1 + \frac{1}{{\min \left\{ {p,q} \right\}}} $.  We also provide a more accurate theorem that exhibits nearly linear-exponential convergence, which we refer to as Theorem \ref{DWRYZS}. This theorem will be postponed to Section \ref{DWSECREVISITED}, as it requires additional preparations.

\begin{lemma}\label{DWLEMMA51}
For $ p,q>0 $, there exists some $ \tilde\lambda  = \tilde\lambda \left( {p,q} \right) > 1 $ such that for any $ m \in \mathbb{N}^+ $,
\[{\left\| {{D^m}{w_{p,q}}\left( x \right)} \right\|_{{L^1}\left( {0,1} \right)}} \leqslant {\tilde\lambda ^m}{m^{\left( {1 + \frac{1}{{\min \left\{ {p,q} \right\}}}} \right)m}}.\]	
\end{lemma}
\begin{remark}
This lemma is also pivotal  in establishing  the nearly linear-exponential convergence as demonstrated in  Theorem \ref{DWRYZS}.
\end{remark}
\begin{proof}
Following a similar method in Lemma \ref{SWLEMMA1}, we have that for fixed $ m \in \mathbb{N}^+ $, there exists  some $ \tilde \lambda_1>0 $ {{(independent of $ x $)}} such that
\begin{equation}\label{DMPQ}
	\left| {{D^m}{w_{p,q}}\left( x \right)} \right| \leqslant \frac{{m!}}{{{\tilde\lambda_1 ^m}}}\max \left\{ {{x^{ - m}}\exp \left( { - \frac{1}{2}{x^{ - \min \left\{ {p,q} \right\}}}} \right),{{\left( {1 - x} \right)}^{ - m}}\exp \left( { - \frac{1}{2}{{\left( {1 - x} \right)}^{ - \min \left\{ {p,q} \right\}}}} \right)} \right\}.
\end{equation}
Then it follows from \eqref{DMPQ} and the symmetry that
\begin{align}
	{\left\| {{D^m}{w_{p,q}}\left( x \right)} \right\|_{{L^1(0,1)}}} &\leqslant \frac{{2m!}}{{\tilde\lambda _1^m}}\int_0^{\frac{1}{2}} {{x^{ - m}}\exp \left( { - \frac{1}{2}{x^{ - \min \left\{ {p,q} \right\}}}} \right)dx} \notag \\
	&  = \frac{{2m!}}{{\tilde\lambda _1^m}}\frac{{{2^{\frac{{m - 1}}{{\min \left\{ {p,q} \right\}}}}}}}{{\min \left\{ {p,q} \right\}}}\int_{{2^{\min \left\{ {p,q} \right\} - 1}}}^{ + \infty } {{y^{\frac{{m - 1}}{{\min \left\{ {p,q} \right\}}} - 1}}{e^{ - y}}dy} \notag \\
	\label{DWPQ2}& \leqslant \frac{2{\Gamma(m)}}{{\tilde\lambda _1^m}}\frac{{{2^{\frac{{m - 1}}{{\min \left\{ {p,q} \right\}}}}}}}{{\min \left\{ {p,q} \right\}}}\Gamma \left( {\frac{{m - 1}}{{\min \left\{ {p,q} \right\}}}} \right),
\end{align}
where $ \Gamma(x) $ is the standard Gamma function. In view of Stirling's approximation $ \Gamma \left( x \right) \sim \sqrt {2\pi } {x^{x - \frac{1}{2}}}{e^{ - x}}$ as  $x \to  + \infty $,
the right side of \eqref{DWPQ2} admits an equivalent form as
\begin{align*}
	&\frac{2{\sqrt {2\pi m} {m^m}{e^{ - m}}}}{{\lambda _1^m}}\frac{{{2^{\frac{{m - 1}}{{\min \left\{ {p,q} \right\}}}}}}}{{\min \left\{ {p,q} \right\}}}\sqrt {2\pi } {\left( {\frac{{m - 1}}{{\min \left\{ {p,q} \right\}}}} \right)^{\frac{{m - 1}}{{\min \left\{ {p,q} \right\}}} - \frac{1}{2}}}{e^{ - \frac{{m - 1}}{{\min \left\{ {p,q} \right\}}}}}\\
	= &{\tilde\lambda _3}\tilde\lambda _2^m{m^{\left( {1 + \frac{1}{{\min \left\{ {p,q} \right\}}}} \right)m - \frac{1}{{\min \left\{ {p,q} \right\}}}}},
\end{align*}
provided with
\[{\tilde \lambda _2} = \frac{1}{{{\tilde \lambda _1}\min \left\{ {p,q} \right\}}}{2^{\frac{1}{{\min \left\{ {p,q} \right\}}}}}{e^{ - \left( {1 + \frac{1}{{\min \left\{ {p,q} \right\}}}} \right)}},\;\;{\tilde \lambda _3} = \frac{{{2^{2 - \frac{1}{{\min \left\{ {p,q} \right\}}}}}\pi }}{{\min \left\{ {p,q} \right\}}}{\left( {\frac{1}{{\min \left\{ {p,q} \right\}}}} \right)^{ - \frac{1}{{\min \left\{ {p,q} \right\}}} - \frac{1}{2}}}.\]
Therefore, it is evident that there exists some $ \tilde\lambda  = \tilde\lambda \left( {p,q} \right) > 1 $ such that
\[{\left\| {{D^m}{w_{p,q}}\left( x \right)} \right\|_{{L^1}\left( {0,1} \right)}} \leqslant {\tilde\lambda ^m}{m^{\left( {1 + \frac{1}{{\min \left\{ {p,q} \right\}}}} \right)m}}.\]
\end{proof}

Regarding  question (Q2), it is somewhat interesting that ones could derive  the following theorem, at least for $ p,q\geqslant 1 $:

\begin{theorem}
For $ p,q\geqslant1 $, the quantitative results  in Theorems \ref{SWT1} and \ref{DWCORO2} remain valid when replacing the weighting function $ w(x) $ with $ w_{p,q}(x) $.
\end{theorem}
\begin{proof}
Recalling the proof of Lemma \ref{SWLEMMA1} and utilizing \eqref{DMPQ}, we have 
\begin{align*}
{\left\| {\left| {{D^m}{w_{p,q}}\left( y \right)} \right|{e^{ - \lambda Ny}}} \right\|_{{L^1}\left( {0,1} \right)}} &\leqslant \frac{{2m!}}{{\tilde \lambda _1^m}}\int_0^{\frac{1}{2}} {{y^{ - m}}\exp \left( { - \frac{1}{2}{y^{ - \min \left\{ {p,q} \right\}}} - \lambda Ny} \right)dy} \\
& \leqslant \frac{{2m!}}{{\tilde \lambda _1^m}}\int_0^{\frac{1}{2}} {{y^{ - m}}\exp \left( { - \frac{1}{2}{y^{ - 1}} - \lambda Ny} \right)dy} 
\end{align*}
due to $ y \geqslant {y^{\min \left\{ {p,q} \right\}}} $ on $ (0,1/2) $.
Therefore, \eqref{SWJIFEN1} holds whenever we replace by $ \lambda_1 $ with $ \tilde \lambda _1 $. Consequently,  Lemma \ref{SWLEMMA1} continues to be applicable to  $ w_{p,q}(x) $, and Theorems \ref{SWT1} and \ref{DWCORO2} are also the case.
\end{proof}

\subsection{Exponential weighting function with a width parameter $ \gamma $}
Finally, we would like to mention the weighting function utilized  by Duignan and Meiss in \cite{MR4582163}, namely a various version of $ w(x) $ with a  \textit{width parameter} $ \gamma>0 $:
\[	\tilde{w}_\gamma\left( x \right):= \left\{ \begin{array}{ll}
	{\left( {\int_0^1 {\exp \left( { - \gamma{s^{ - 1}}{{\left( {1 - s} \right)}^{ - 1}}} \right)ds} } \right)^{ - 1}}\exp \left( { - \gamma{x^{ - 1}}{{\left( {1 - x} \right)}^{ - 1}}} \right),&x \in \left( {0,1} \right),  \hfill \\
	0,&x \notin \left(0,1\right). \hfill \\
\end{array}  \right.\]
They compared numerically the impact of different width parameters $ \gamma $ on the convergence rate for the two-wave Hamiltonian system (e.g., to illustrate the converse KAM method that detects the breakup of tori \cite{MR4198485,MR4322369}), and they observed that, for different cases (such as varying orbits), it is difficult to obtain a uniform optimal $ \gamma $ to accelerate the convergence rate. Therefore, it is necessary to theoretically provide quantitative exponential convergence rates for all $ \gamma $. It should be noted that even in the analysis of weighted Birkhoff averages, estimating the convergence rate of the weighting function $ \tilde w_\gamma (x)$ is somewhat difficult in the previous techniques (see \cite{MR4768308} for instance), let alone \textit{quantitatively}.

Here we do not intend to delve into detailed discussions, but from the results, the difference with $ w(x) $ lies in, for example, the parameter $ \xi= \sqrt {2\lambda }  + {e^{ - 1}}\sqrt {{\lambda ^2} + \left\| \rho  \right\|_{\mathbb{T}}^2}  >0 $ in the exponential convergence rate of Theorem \ref{SWT1} being replaced by $ \xi_\gamma = \sqrt {2\gamma\lambda }  + {e^{ - 1}}\sqrt {{\lambda ^2} + \left\| \rho  \right\|_{\mathbb{T}}^2}  >0 $ (noting only that the exponential part in \eqref{SWJIFEN1} in Lemma \ref{SWLEMMA1}  becomes $ \exp \left( { - \gamma u/2 - \lambda N{u^{ - 1}}} \right) $, which yields the exponential part $ \exp \left( { - \sqrt {2\gamma \lambda N} } \right) $ in \eqref{DWLEMMA21} with the weighting function $ w_\gamma(x) $). 

It is worth emphasizing that increasing $ \gamma $ cannot make practical computations faster \textit{indefinitely}, as the control coefficient for the exponential convergence will \textit{inevitably} tend to infinity, leading to lack of uniformity (otherwise, when $ \gamma \to +\infty $, the weighted average tends to $ 0 $ rather than the spatial average). This is consistent with the views of Duignan and Meiss in \cite{MR4582163}. However, for a given system, we can provide a feasible approach to the unsolved problem of finding the optimal width parameter $ \gamma $ in \cite{MR4582163}, by quantitatively estimating the exponential convergence rate (calculating the control coefficient quantitatively as well), thereby choosing the optimal $ \gamma $ to \textit{minimize} the upper bound of the  error.

{{We end this section by mentioning the ``Planck-taper'' window function which appears to have attracted recent attention (see, for instance, Dudal et al. \cite{MR4703229}). By utilizing the ideas presented in this paper, one can establish the corresponding version of Lemma \ref{DWLEMMA51}. This allows us to obtain some quantitative results similar to Theorems \ref{SWT1}, \ref{DWCORO2} and \ref{DWJQ} for this weighted average.}}

\section{Further applicability: the continuous case and the general decaying case}\label{DWSEC5}
\subsection{The continuous case}
We first would like to emphasize that all results in this paper can be extended to the continuous case, i.e., in a integral form. The proofs become simpler, as we do not need to employ the Poisson summation formula, and thus do not require the elaborate summation estimates as in Lemma \ref{SWLEMMA2}. Additionally, it is worth mentioning that the uniform exponential convergence rates obtained will be faster in certain cases, as in those cases, the  $ \|\rho\|_\mathbb{T} $ term in the conclusions (see Theorem \ref{SWT1} for instance) will be replaced by $ |\rho| $ (note that $ |\rho| >\|\rho\|_\mathbb{T} $ for all $ |\rho|>2\pi $), eliminating the need for the Poisson summation formula. Another distinction in terms of the conclusions between the continuous and discrete cases is the $ \mathcal{O}^{\#}(T^{-1}) $ polynomial convergence condition for the unweighted type. Define in a similar way,
\[{{\rm{DW}}_T^{\rm con}}\left( {\lambda ,\rho,\theta  } \right): = \frac{1}{T}\int_0^T {{e^{ - \lambda t}}\sin \left( {\theta  + \rho t} \right)dt} , \]
as well as the $ w $-weighted type, 
\[{\rm{WDW}}_T^{\rm con}\left( {\lambda ,\rho ,\theta } \right): = \frac{1}{T}\int_0^T {w\left( {t/T} \right){e^{ - \lambda t}}\sin \left( {\theta  + \rho t} \right)dt} .\]
Then, $ {\rm{WDW}}_T^{\rm con}\left( {\lambda ,\rho ,\theta } \right) $ admits uniform exponential convergence as  mentioned previously. However, through  simple calculations, we can prove that $ \left| {{{\rm{DW}}_T^{\rm con}}\left( {\lambda ,\rho,\theta  } \right)} \right| = {\mathcal{O}^\# }\left( {{T^{ - 1}}} \right)$ as $T \to  + \infty $     if $ \lambda \sin \theta  + \rho \cos \theta  \ne 0 $, which differs  from (II) in Theorem \ref{SWT1}. To enable readers to understand the aforementioned analysis more clearly, we present the following  theorem as a continuous form of Theorem \ref{SWT1} and provide a brief outline of its proof.

\begin{theorem}\label{DWLXFJDL}
(I) For any fixed $ \lambda  > 0 $ and $ \rho  \in \mathbb{R} $, 
\[\left| {\rm{WDW}}_T^{\rm con}\left( {\lambda ,\rho ,\theta } \right) \right| = \mathcal{O}\left( {\exp \left( { - {\xi _{\rm{con}}}\sqrt T } \right)} \right),\;\;T \to  + \infty \]
uniformly holds with $ \theta  \in \mathbb{R} $, where $ \xi_{\rm{con}}=\xi_{\rm{con}}(\lambda,\rho) = \sqrt {2\lambda }  + {e^{ - 1}}\sqrt {{\lambda ^2} + | \rho |^2}  >0 $.  In particular, the control coefficient is independent of $ \rho $, and is uniformly bounded for large $ \lambda $.
\\
(II) However, for any fixed $ \lambda  > 0 $ and $ \rho,\theta  \in \mathbb{R} $,  there exists infinitely many $ N \in \mathbb{N}^+ $ such that 
\[\left| {{{\rm{DW}}_T^{\rm con}}\left( {\lambda ,\rho,\theta  } \right)} \right| = {\mathcal{O}^\# }\left( {{T^{ - 1}}} \right),\;\;T \to  + \infty,\]
whenever $ \lambda \sin \theta  + \rho \cos \theta  \ne 0 $.
\end{theorem}
\begin{proof}
We only show the strategy of (I).  Unlike the discrete case, which requires the application of the Poisson summation formula, we only need to perform integration by parts on the weighted integral $ {\rm{WDW}}_T^{\rm con}\left( {\lambda ,\rho ,\theta } \right) $, with variable integration time denoted as $ m $. Therefore, this indeed corresponds to a segment of the discrete case; specifically, it does not involve $ 2 \pi i n $ within \eqref{SWZJ2}.    Consequently, fundamentally this resembles Lemma \ref{SWLEMMA2} without necessitating partitioning of summations; one simply needs to select $ m \sim \lambda_3^{-1}{e^{-1}}\sqrt{N}  $ as $ N \to +\infty  $ to complete the  proof of (I).
\end{proof}

Next, we further discuss analogues of Theorem \ref{DWCORO2}. Similarly, in the continuous case, as long as the observable belongs to $ A(\mathbb{T}) $, the corresponding time average along the weighted decaying wave admits uniform exponential convergence. This is in stark contrast to the unweighted Birkhoff integrals. See 	Forni \cite{MR1888794} for instance, given an irrational rotation 
$ \mathscr{T}_\rho^t:\theta\to \theta+t\rho \mod 1$ in each coordinate on $ \mathbb{T}^2 $ observed by an observable $ f $ with bounded variation, the upper bound for \textit{all $ \theta \in \mathbb{T}^2 $} of the difference between the time average $ {T^{ - 1}}\int_0^T {f\left( {\mathscr{T}_\rho ^t\left( \theta  \right)} \right)dt}   $ and the spatial average $ \int_\mathbb{T} {f\left( x \right)dx}  $ can often only be estimated as 
$ \mathcal{O}\left( {{T^{ - 1}}\ln T{{\left( {\ln \ln T} \right)}^{1+\alpha} }} \right) $ using the Denjoy-Koksma inequality and the metric theory of continued fractions, where $ \alpha>0 $ is arbitrarily fixed. See also Dolgopyat and Fayad \cite{MR4179835} for a  more precise description of this aspect. This bound includes an additional logarithmic term beyond the optimal 
$ \mathcal{O}^{\#}(T^{-1}) $ convergence rate achievable in sufficiently smooth cases, through a Fourier argument based on the co-homological equation.  Consequently, our results are somewhat \textit{unexpected}, as we achieve uniform exponential convergence from observables with very weak regularity.
\subsection{The general decaying case}

Finally, let us discuss {{weighted averages}} along general decaying waves. It should be noted that the approaches in this paper \textit{are not limited to} weighted averages of exponentially decaying waves, but also extend to more general decay, such as the polynomial decay, exponential plus polynomial decay,  super-exponential decay, etc. Specifically, the exponential plus polynomial decay, exemplified by $e^{-\lambda x} x^\iota\sin (\rho x)$ where $\lambda >0$ and $\iota, \rho \in \mathbb{R}$, has a distinctive background as elucidated in the Introduction; it serves as a fundamental solution of the linear differential equation $x' = Ax $, where  the constant matrix $A$ only admits eigenvalues within the unit circle $ \mathbb{S}^1 $--or equivalently, with negative maximal Lyapunov exponents,  and the degree $\iota$ of the polynomial   part $x^\iota$ depends on the multiplicity of the eigenvalues of $A$. Furthermore, the techniques developed in this paper are applicable to  {{weighted averages}} along  trajectories of certain \textit{nonlinear} systems that can be \textit{smoothly conjugated}  to linear ones, e.g., 
\[x' = Ax +f(x)\;\;\text{and}\;\;x_{n+1} =Ax_n+f(x_n),\]
although in such cases, one must  investigate the regularity of the conjugations.  If only arbitrary polynomial convergence rates are desired when the conjugation is $ C^\infty $, the analysis will be straightforward--indeed, a slight modification of \cite{MR3755876, MR3718733,MR4768308,TLCesaro,TLmultiple} or the approaches presented in this paper can achieve the desired conclusion; however, for the \textit{quantitative} exponential convergence rates as considered in this paper (which will change form for more general decay), some essential estimates need to be established as in Lemma \ref{SWLEMMA1}. In order to highlight our approaches more clearly, we prefer to focus on the exponentially decaying waves in this paper, without delving into the details of these more general cases, as they do not differ significantly in terms of conceptual framework. For this reason, we only present the following \textit{qualitative} exponential convergence theorem with a brief outline of the proof.
\setcounter{footnote}{0}
\renewcommand{\thefootnote}{\fnsymbol{footnote}}
\begin{theorem}
(I) Consider the continuous nonlinear dynamical   system
\begin{equation}\label{DWF1}
	x' = F\left( x \right),\;\;  x \in \mathbb{R}^{{d}},\;\; {{d \in \mathbb{N}^+,}} 
\end{equation}
 where $ F:{\mathbb{R}^{{d}}} \to {\mathbb{R}^{{d}}} $ satisfies $ F\left( O \right) = O $,   and all eigenvalues of $ DF\left( O \right) $ lie inside the unit circle $ \mathbb{S}^1 $. Assume there exists a neighbourhood $ \Omega  $ of $ O $ and a  conjugation $ \Phi  \in {C^2}( {\overline \Omega  } ) $ such that $ \Phi \left( O \right) = O $, $ \det D\Phi \left( O \right) \ne 0 $, and that $ \Phi  \circ x\left( t \right) = y\left( t \right) \circ \Phi  $, where $ y $ is the flow of $ y' = DF\left( O \right)y $ with the same initial of \eqref{DWF1}. Then for any flow $ {x\left( {t,{x_0}} \right)} $ starting from the initial point $x_0 \in \Omega $, the  {{weighted averages}}
\[\frac{1}{{{A_N}}}\sum\limits_{n = 0}^{N - 1} {w\left( {n/N} \right)x\left( {n,{x_0}} \right)} ,\;\;\frac{1}{T}\int_0^T {w\left( {t/T} \right)x\left( {t,{x_0}} \right)dt} \]
converge uniformly (with respect to $ x_0 $) and exponentially\footnote{{{Although the eigenvalues may be multiple in this case, by making some concessions in the exponential rate (to eliminate the polynomial part), one can obtain a quantitative exponential rate similar to that in (II). For the sake of brevity, we only provide here a qualitative description.}}}  to $ O $.
\\
(II) Consider the discrete nonlinear dynamical system
\[{x_{n + 1}} = F\left( {{x_n}} \right),\;\;  x \in \mathbb{R}^{{d}},\;\; {{d \in \mathbb{N}^+,}}\]
 where $ F:{\mathbb{R}^{{d}}} \to {\mathbb{R}^{{d}}} $ satisfies $ F\left( O \right) = O $,   and all eigenvalues $ {s _j} $ ($ j=1,\ldots,{{d}} $) of $ DF\left( O \right) $ are simple and lie inside the unit circle $ \mathbb{S}^1 $. Moreover, $ F $ is locally $ C^\infty $, and the following nonresonant condition holds:
\[{s_j} \ne s_1^{{m_1}} \cdots s_{{d}}^{{m_{{d}}}},\;\;1 \leqslant j \leqslant {{d}},\;\; \left( {{m_1}, \ldots ,{m_{{d}}}} \right) \in {\mathbb{N}^{{d}}},\;\;\sum\nolimits_{i = 1}^{{d}} {{m_i}}  \geqslant 2.\]
   Then there exists a neighbourhood $ \Omega  $ of $ O $, such that for any iterated sequence $ {\left\{ {{x_n}} \right\}_{n \in \mathbb{N}}}$ starting from the initial point $x_0 \in \Omega $, the   {{weighted average}}
 \[\frac{1}{{{A_N}}}\sum\limits_{n = 0}^{N - 1} {w\left( {n/N} \right){x_n}} \]
 converges uniformly (with respect to $ x_0 $) and exponentially  to $ O $. In particular, the convergence rate is $ \mathcal{O}(\exp ( - {\xi _A}\sqrt N )) $, where $ {\xi _A} = \sqrt {2{{\min }_{1 \leqslant j \leqslant {{d}}}}{s_j}}  + {e^{ - 1}}{\min _{1 \leqslant j \leqslant {{d}}}}{s_j}>0 $.
\end{theorem}
\begin{proof}
We only prove (II), as the analysis for (I) is similar. For convenience, we only consider the case where $ {{d}}=2 $. By utilizing Sternberg's linearization theorem \cite{MR0096853,MR0096854}, we obtain a local $ C^2 $ diffeomorphism $ \Phi  $ such that $ \Phi  \circ F = \Lambda  \circ \Phi  $, with $ \Phi \left( O \right) = O $ and $ D\Phi \left( O \right) = \mathbb{I} $. Then it is evident that for any $ x_0 $ sufficiently close to $ O $, it holds
\begin{align*}
	{x_n}& = F\left( {{x_{n - 1}}} \right) = {\Phi ^{ - 1}} \circ \Lambda  \circ \Phi \left( {{x_{n - 1}}} \right) = {\Phi ^{ - 1}} \circ \Lambda  \circ \Phi  \circ F\left( {{x_{n - 2}}} \right)\\
	& = {\Phi ^{ - 1}} \circ \Lambda  \circ \Phi  \circ {\Phi ^{ - 1}} \circ \Lambda  \circ \Phi \left( {{x_{n - 2}}} \right) =  \cdots  = {\Phi ^{ - 1}}{\Lambda ^n}\Phi \left( {{x_0}} \right).
\end{align*}
Generally,  we set $ \Phi \left( x \right) = Ax + B{x^2} + o( {{{\left| x \right|}^2}} ) $ with $ x: = {\left( {{x_1},{x_2}} \right)^\top} $ and $ {x^2}: = {\left( {x_1^2,{x_1}{x_2},x_2^2} \right)^\top} $ (indeed, $ A=\mathbb{I} $ in the discrete case here, but it may vary in the continuous case), then one observes that $ {\Phi ^{ - 1}}\left( x \right) = {A^{ - 1}}x + {A^{ - 1}}B\tilde A{x^2} + o( {{{\left| x \right|}^2}} ) $ (this is actually a generalized version of the inverse Lagrange theorem), where 
\[A: = \left( {\begin{array}{*{20}{c}}
		{{A_{11}}}&{{A_{12}}} \\ 
		{{A_{21}}}&{{A_{22}}} 
\end{array}} \right),\;\;{A^{ - 1}}: = \left( {\begin{array}{*{20}{c}}
		{{a_{11}}}&{{a_{12}}} \\ 
		{{a_{21}}}&{{a_{22}}} 
\end{array}} \right),\;\;B: = \left( {\begin{array}{*{20}{c}}
		{{B_{11}}}&{{B_{12}}}&{{B_{13}}} \\ 
		{{B_{21}}}&{{B_{22}}}&{{B_{23}}} 
\end{array}} \right),\]
and
\[\tilde A: = \left( {\begin{array}{*{20}{c}}
		{a_{11}^2}&{2{a_{11}}{a_{12}}}&{a_{22}^2} \\ 
		{{a_{11}}{a_{21}}}&{{a_{11}}{a_{22}} + {a_{12}}{a_{21}}}&{{a_{12}}{a_{22}}} \\ 
		{a_{21}^2}&{2{a_{21}}{a_{22}}}&{a_{22}^2} 
\end{array}} \right).\]
This leads to 
\[\left| {\frac{1}{{{A_N}}}\sum\limits_{n = 0}^{N - 1} {w\left( {n/N} \right){x_n}}  - \frac{1}{{{A_N}}}\sum\limits_{n = 0}^{N - 1} {w\left( {n/N} \right){A^{ - 1}}{\Lambda ^n}\Phi \left( {{x_0}} \right)} } \right| \leqslant \frac{K}{{{A_N}}}\sum\limits_{n = 0}^{N - 1} {w\left( {n/N} \right){{\left| {{\Lambda ^n}\Phi \left( {{x_0}} \right)} \right|}^{2n}}} \]
for some universal $ K>0$. Both $ \frac{1}{{{A_N}}}\sum\nolimits_{n = 0}^{N - 1} {w\left( {n/N} \right){A^{ - 1}}{\Lambda ^n}\Phi \left( {{x_0}} \right)}  $ and $ \frac{K}{{{A_N}}}\sum\nolimits_{n = 0}^{N - 1} {w\left( {n/N} \right){{\left| {{\Lambda ^n}\Phi \left( {{x_0}} \right)} \right|}^{2n}}} $ converge uniformly (with respect to $ x_0 $) and exponentially  to $ O $, as the components of $ {{A^{ - 1}}{\Lambda ^n}\Phi \left( {{x_0}} \right)} $ and $ {{{\left| {{\Lambda ^n}\Phi \left( {{x_0}} \right)} \right|}^{2n}}} $ are exponentially decaying waves which we have discussed in Theorems \ref{SWT1} and \ref{DWCORO2}. To be more precise, from the nonresonant condition for eigenvalues, we have  $ s_j \ne 0$ for $ 1\leqslant j \leqslant {{d}}$. Consequently, the smallest decaying parameter of such waves is $ {{{\min }_{1 \leqslant j \leqslant {{d}}}}{s_j}} $, which leads to the $ \mathcal{O}(\exp ( - {\xi _A}\sqrt N )) $ exponential convergence of $ \frac{1}{{{A_N}}}\sum\nolimits_{n = 0}^{N - 1} {w\left( {n/N} \right){x_n}} $ with $ {\xi _A} = \sqrt {2{{\min }_{1 \leqslant j \leqslant {{d}}}}{s_j}}  + {e^{ - 1}}{\min _{1 \leqslant j \leqslant {{d}}}}{s_j}>0 $,  similar to the proof of (III) in Theorem \ref{DWCORO2}.   As for the continuous case, $ {{A^{ - 1}}{\Lambda ^n}\Phi \left( {{x_0}} \right)} $ may be more complicated as $ \Lambda  $ may contain Jordan blocks, however, the components of $ {{A^{ - 1}}{\Lambda ^n}\Phi \left( {{x_0}} \right)} $ are decaying waves  with  exponential plus polynomial decay, hence the weighted   average along it still converges exponentially since the polynomial parts could be well dominated by the exponential parts in the proof. 
\end{proof}

We end this section by emphasizing the followings: (i) Both two cases in (I) and (II) could be extended to the smoothly-observed case as discussed in Theorem \ref{DWCORO2}; (ii) the existence of the conjugation in (I) can be guaranteed by the classical Hartman-Grobman theorem, but the regularity is generally at most of H\"older type below $ C^1 $, regardless of how high the regularity of $ F $ is, as detailed in Arnold's book \cite{MR1162307}; (iii) The regularity of $ F $ in (II) allows for a finitely differentiable type depending on the nonresonant conditions of the eigenvalues, see Sternberg \cite{MR0096853,MR0096854}. 
For a more explicit exposition on Sternberg's result, we  would also like to mention the work of Zhang et al. \cite{MR3632558}.

\section{Exponential acceleration in analytic  quasi-periodic dynamical systems revisited}\label{DWSECREVISITED}
This section is mainly devoted to improve a previous \textit{qualitative} result obtained by the authors in \cite{MR4768308} to a \textit{quantitative} one. Consider a quasi-periodic weighted Birkhoff average
\[\frac{1}{{{A_N}}}\sum\limits_{n = 0}^{N - 1} {w\left( {n/N} \right)f\left( {{\mathscr{T}^n_\rho}\left(\theta\right)} \right)} \]
 as that in \eqref{DWWBA}, where the $ d $-torus is modified to $ {\mathbb{T}^d}: = {\mathbb{R}^d}/{\mathbb{Z}^d} $ with $ d \in \mathbb{N}^+ $ for brevity, the initial point $ \theta \in \mathbb{T}^d $, the quasi-periodic mapping is specified  by $ {\mathscr{T}_\rho }\left( \theta  \right) = \theta  + \rho \bmod 1 $ with the nonresonant rotation vector $ \rho \in \mathbb{T}^d $, and $ f $ is a smooth observable on $ \mathbb{T}^d $. One of the most important (which we say universal) results from \cite{MR4768308}, Corollary 3.1,  states that if $ f $ is \textit{analytic}, then for \textit{almost all} (in the sense of full Lebesgue measure) rotation vectors $ \rho $, the weighted Birkhoff error term
\begin{equation}\label{DWERROR}
	{\bf Error}_{N}: = \left|\frac{1}{{{A_N}}}\sum\limits_{n = 0}^{N - 1} {w\left( {n/N} \right)f\left( {{\mathscr{T}^n_\rho}\left(\theta\right)} \right)} - \int_{{\mathbb{T}^d}} {f\left( x \right)dx}\right| 
\end{equation}
exhibits  \textit{qualitative} uniform (with respect to $ \theta $) exponential convergence $ \mathcal{O}\left( {\exp \left( { - c{N^\vartheta }} \right)} \right) $, where $ c, \vartheta >0$  are certain universal constants. It is worth noting that, for \textit{qualitative} considerations, \cite{MR4768308} provides a strict lower bound for $ \vartheta $ as $ 1/\left( {d+12 } \right) $, although this is not explicitly stated there. Our motivation in the forthcoming Theorem \ref{DWJQ} is to refine this bound and demonstrate   a variety of representative  scenarios,  utilizing  the innovative  techniques developed throughout this paper. We emphasize that if one is only interested in arbitrary polynomial convergence, one can follow the ideas presented in Das and Yorke \cite{MR3755876} (or Das et al. \cite{MR3718733}), making the complicated approaches employed here entirely unnecessary. The main difficulty lies in properly dealing with  small divisors to \textit{quantify}  the exponential convergence, which still remains  \textit{unexplored}.
\setcounter{footnote}{0}
\renewcommand{\thefootnote}{\fnsymbol{footnote}}
\begin{theorem}\label{DWJQ}
	Consider the weighted Birkhoff error term $ {\bf Error}_{N} $ in \eqref{DWERROR}.
	\begin{itemize}
\item[(I)] If $ f $ is analytic, then for almost all rotations $ \rho $ and all $ \zeta>1 $, we have 
\[{\bf Error}_{N}= \mathcal{O}\left( {\exp \left( { - {c_{\rm I}}{N^{\frac{1}{{d + 2}}}}{{\left( {\ln N} \right)}^{ - \frac{\zeta}{{d + 2}}}}} \right)} \right),\]
 where $ c_{\rm I} >0$ is some universal constant.
		
\item[(II)] If $ f $ is analytic, then for   rotations $ \rho $ with Diophantine exponent of $ d $\footnote{It is well known that such rotations  only form a set of zero Lebesgue measure in $ \mathbb{R}^d $.}, namely
\[\left| {k \cdot \rho  - n} \right| > \frac{C}{{{{\left| k \right|}^d}}},\;\;C > 0,\;\;\forall 0 \ne k \in {\mathbb{Z}^d},\;\;\forall n \in \mathbb{Z},\]
we have
\[{\bf Error}_{N}=\mathcal{O}\left( {\exp \left( { - {c_{\rm II}}{N^{\frac{1}{{d + 2}}}}} \right)} \right),\]
 where $ c_{\rm II} >0$ is some universal constant.
\item[(III)]   If $ f $ is a finite trigonometric polynomial, then for almost all rotations,  we have 
\[ {\bf Error}_{N}=\mathcal{O}\left( {\exp \left( { - {c_{\rm III}}\sqrt N } \right)} \right),\]
where $ c_{\rm III} >0$ is some universal constant.
	\end{itemize}
\end{theorem}
\begin{proof}
To begin, we provide a thorough demonstration for (I), serving as the foundation for the proofs of (II) to (III).

Denote by $ \mathbb{T}_r^d: = \left\{ {z = u + iv:u \in {\mathbb{T}^d},\left| v \right| \leqslant r} \right\} $ the thickened torus, and define the norm $ {\left\| f \right\|_r}: = \mathop {\sup }\nolimits_{\left| v \right| \leqslant r} {\left( {\int_{{\mathbb{T}^d}} {{{\left| {f\left( {u + iv} \right)} \right|}^2}du} } \right)^{\frac{1}{2}}} $. Then,  it is well known that if $ f $ is analytic,  its Fourier coefficients satisfy $ |{{\hat f}_k}| \leqslant {\left\| f \right\|_r}{e^{ - 2\pi r\left| k \right|}} $ for all $ { k \in {\mathbb{Z}^d}} $, where $ r>0 $ is the analytic radius of $ f $, see Salamon \cite{MR2111297} for instance. Next, for any fixed $ \zeta>1 $, we establish a nonresonant condition from Herman \cite{MR0538680}  that is satisfied by almost all rotations $ \rho \in \mathbb{R}^d $ (as $ \alpha $ varies):
\begin{equation}\label{DWALMOST}
	\left| {k \cdot \rho  - n} \right| > \frac{\alpha }{{{{\left| k \right|}^{d }}{{\ln }^\zeta }\left( {1 + \left| k \right|} \right)}},\;\;\alpha  > 0,\;\;\forall 0 \ne k \in {\mathbb{Z}^d},\;\;\forall n \in \mathbb{Z}.
\end{equation}
Below we estimate the Birkhoff error term $ {\bf Error}_{N} $ in (I).

 Note $ \int_{{\mathbb{T}^d}} {f\left( x \right)dx}  = \frac{1}{{{A_N}}}\sum\nolimits_{n = 0}^{N - 1} {w\left( {n/N} \right){{\hat f}_0}}  $. Then it follows that 
\begin{align}
\frac{1}{{{A_N}}}\sum\limits_{n = 0}^{N - 1} {w\left( {n/N} \right)f\left( {\mathscr{T}_\rho ^n\left( \theta  \right)} \right)}  - \int_{{\mathbb{T}^d}} {f\left( x \right)dx}  &=\sum\limits_{0 \ne k \in {\mathbb{Z}^d}} {{{\hat f}_k}{e^{2\pi ik\cdot\theta }}\frac{1}{{{A_N}}}\sum\limits_{n = 0}^{N - 1} {w\left( {n/N} \right){e^{2\pi ink \cdot \rho }}} } \notag \\
& = \sum\limits_{0 \ne k \in {\mathbb{Z}^d},\left| k \right| < \mathcal{K}\left( N \right)} {{{\hat f}_k}{e^{2\pi ik\cdot\theta }}\frac{1}{{{A_N}}}\sum\limits_{n = 0}^{N - 1} {w\left( {n/N} \right){e^{2\pi ink \cdot \rho }}} } \notag \\
& + \sum\limits_{0 \ne k \in {\mathbb{Z}^d},\left| k \right| \geqslant \mathcal{K}\left( N \right)} {{{\hat f}_k}{e^{2\pi ik\cdot\theta }}\frac{1}{{{A_N}}}\sum\limits_{n = 0}^{N - 1} {w\left( {n/N} \right){e^{2\pi ink \cdot \rho }}} } \notag \\
\label{DWFENDUAN}: &= {\mathscr{S}_{\rm T}} + {\mathscr{S}_{\rm R}},
\end{align}
where the truncated order $ \mathcal{K}\left( N \right) $ will be specified later. 

For the truncated term $ {\mathscr{S}_{\rm T}} $, we first estimate $  \left| {\frac{1}{{{A_N}}}\sum\nolimits_{n = 0}^{N - 1} {w\left( {n/N} \right){e^{2\pi ink \cdot \rho }}} } \right|$. With the Poisson summation formula, we obtain
\begin{align}
\left|\frac{1}{{{A_N}}}\sum\limits_{n = 0}^{N - 1} {w\left( {n/N} \right){e^{2\pi ink \cdot \rho }}} \right|		& = \frac{1}{{{A_N}}}\left|\sum\limits_{n =  - \infty }^\infty  {\int_{ - \infty }^{ + \infty } {w\left( {s/N} \right){e^{2\pi i\left( {k \cdot \rho  - n} \right)s}}ds} }\right| \notag \\
	&\leqslant \frac{N}{{{A_N}}}\sum\limits_{n =  - \infty }^\infty  \left|{\int_0^1 {w\left( z \right){e^{2\pi Ni\left( {k \cdot \rho  - n} \right)z}}dz} }\right| \notag \\
\label{DWIIIII}	& = \frac{N}{{{A_N}}}\left( \left|{\int_0^1 {w\left( z \right){e^{2\pi Ni\left( {k \cdot \rho  - {n_k}} \right)z}}dz} }\right| + \sum\limits_{n \ne {n_k}} \left|{\int_0^1 {w\left( z \right){e^{2\pi Ni\left( {k \cdot \rho  - n} \right)z}}dz} }\right|  \right), 
\end{align}
where $ {n_k}: = {\inf _{n \in \mathbb{Z}}}\left| {k \cdot \rho  - n} \right| $ for fixed $ k $ and $ \rho $. For any truncated order $ 2\leqslant  \mathcal{K}\left( N \right) = \mathcal{O}( {{N^{\frac{1}{{d+1}}}}} ) $ with $\mathop {\lim }\nolimits_{N \to  + \infty }  \mathcal{K}(N)= +\infty $, {{set the number of integrations by parts $ m_N $}} as
\[{m_N} \sim \frac{1}{e}{\left( {\frac{{2\pi \alpha N}}{{\tilde \lambda \left( {1,1} \right)\mathcal{K}{{\left( N \right)}^{d }}{{\ln }^\zeta }\left( {1 + \mathcal{K}\left( N \right)} \right)}}} \right)^{\frac{1}{2}}}\geqslant 2 , \]
where $ \tilde{\lambda}(1,1) $ is the constant given in Lemma \ref{DWLEMMA51}. Then we have

\begin{align}
	\left| {\int_0^1 {w\left( z \right){e^{2\pi Ni\left( {k \cdot \rho  - {n_k}} \right)z}}dz} } \right| & \leqslant \frac{{{{\left\| {{D^{{m_N}}}w} \right\|}_{{L^1}\left( {0,1} \right)}}}}{{{{\left( {2\pi N\left| {k \cdot \rho  - {n_k}} \right|} \right)}^{{m_N}}}}}\notag \\
\label{DWS11}&  \leqslant {\left( {\frac{{\tilde \lambda \left( {1,1} \right){{\left| k \right|}^{d }}{{\ln }^\zeta }\left( {1 + \left| k \right|} \right)m_N^2}}{{2\pi \alpha N}}} \right)^{{m_N}}} \\
&  \leqslant {\left( {\frac{{\tilde \lambda \left( {1,1} \right)\mathcal{K}{{\left( N \right)}^{d}}{{\ln }^\zeta }\left( {1 + \mathcal{K}\left( N \right)} \right)m_N^2}}{{2\pi \alpha N}}} \right)^{{m_N}}}\notag \\
\label{DWIIII}	& = \mathcal{O}\left( {\exp \left( { - 2e^{-1}\sqrt {\frac{{2\pi \alpha N}}{{\tilde \lambda \left( {1,1} \right)\mathcal{K}{{\left( N \right)}^{d }}{{\ln }^\zeta }\left( {1 + \mathcal{K}\left( N \right)} \right)}}} } \right)} \right),
\end{align}
and similarly, 
\begin{align}
{\sum\limits_{n \ne {n_k}} \left|{\int_0^1 {w\left( z \right){e^{2\pi Ni\left( {k \cdot \rho  - n} \right)z}}dz} }\right| } &\leqslant \sum\limits_{n \ne {n_k}} {\frac{{{{\left\| {{D^{{m_N}}}w} \right\|}_{{L^1}\left( {0,1} \right)}}}}{{{{\left( {2\pi N\left| {k \cdot \rho  - n} \right|} \right)}^{{m_N}}}}}} \notag \\
&  \leqslant \frac{{2{{\left\| {{D^{{m_N}}}w} \right\|}_{{L^1}\left( {0,1} \right)}}}}{{{{\left( {\pi N} \right)}^{{m_N}}}}} + \frac{{{{\left\| {{D^{{m_N}}}w} \right\|}_{{L^1}\left( {0,1} \right)}}}}{{{{\left( {2\pi N} \right)}^{{m_N}}}}}\sum\limits_{n \ne {n_k},{n_k} \pm 1} {\frac{1}{{{{\left| {k \cdot \rho  - n} \right|}^{{m_N}}}}}}  \notag \\
\label{DWIIIIII}& = \mathcal{O}\left( {\exp \left( { - 2e^{-1}\sqrt {\frac{{2\pi \alpha N}}{{\tilde \lambda \left( {1,1} \right)\mathcal{K}{{\left( N \right)}^{d }}{{\ln }^\zeta }\left( {1 + \mathcal{K}\left( N \right)} \right)}}} } \right)} \right), 
\end{align}
because $ \left| {k \cdot \rho  - \left( {{n_k} \pm 1} \right)} \right| \geqslant 1/2 $,   $ {\left\| {{D^{{m_N}}}w} \right\|_{{L^1}\left( {0,1} \right)}}/{\left( {\pi N} \right)^{{m_N}}} $ can be dominated by \eqref{DWS11}, and 
\[\sum\limits_{n \ne {n_k},{n_k} \pm 1} {\frac{1}{{{{\left| {k \cdot \rho  - n} \right|}^{{m_N}}}}}}  \leqslant \sum\limits_{n \ne {n_k},{n_k} \pm 1} {\frac{1}{{{{\left| {k \cdot \rho  - n} \right|}^2}}}}  \leqslant 2\sum\limits_{n = 0}^\infty  {\frac{1}{{{{\left( {n + 1/2} \right)}^2}}}}  <  + \infty. \]
Combining \eqref{DWIIII}, \eqref{DWIIIII}, \eqref{DWIIIIII} and utilizing $ \sum\nolimits_{0 \ne k \in {\mathbb{Z}^d},\left| k \right| < \mathcal{K}\left( N \right)} {| {{{\hat f}_k}{e^{2\pi ik \cdot \theta }}} |}  \leqslant \sum\nolimits_{0 \ne k \in {\mathbb{Z}^d}} {| {{{\hat f}_k}} |}  <  + \infty  $,  $\mathop {\sup }\nolimits_{N \geqslant 3} N/{A_N} = \mathop {\sup }\nolimits_{N \geqslant 3} {\left( {{N^{ - 1}}\sum\nolimits_{n = 0}^{N - 1} {w\left( {n/N} \right)} } \right)^{ - 1}} <  + \infty  $, we arrive at  the estimate for the  truncated term $ {\mathscr{S}_{\rm T}} $ as
\begin{equation}\label{DWTRUNC}
	\left|{\mathscr{S}_{\rm T}}\right| = \mathcal{O}\left( {\exp \left( { - 2e^{-1}\sqrt {\frac{{2\pi \alpha N}}{{\tilde \lambda \left( {1,1} \right)\mathcal{K}{{\left( N \right)}^{d }}{{\ln }^\zeta }\left( {1 + \mathcal{K}\left( N \right)} \right)}}} } \right)} \right).
\end{equation}

As for the remainder term $ \mathscr{S}_{\rm R} $, it is evident that 
\begin{align}
\left| {{\mathscr{S}_{\rm R}}} \right| &\leqslant \sum\limits_{0 \ne k \in {\mathbb{Z}^d},\left| k \right| \geqslant \mathcal{K}\left( N \right)} {|{{\hat f}_k}|\frac{1}{{{A_N}}}\sum\limits_{n = 0}^{N - 1} {w\left( {n/N} \right)} } \notag \\
& = \sum\limits_{0 \ne k \in {\mathbb{Z}^d},\left| k \right|\geqslant \mathcal{K}\left( N \right)} {|{{\hat f}_k}|} \notag \\
& \leqslant \sum\limits_{0 \ne k \in {\mathbb{Z}^d},\left| k \right| \geqslant \mathcal{K}\left( N \right)} {{{\left\| f \right\|}_r}{e^{ - 2\pi r\left| k \right|}}} \notag \\
\label{DWIIIIIII}& = \mathcal{O}\left( {\exp \left( { - r'\mathcal{K}\left( N \right)} \right)} \right), 
\end{align}
provided with any $ 0 < r' < 2\pi r $.

Finally, by choosing the truncated order as 
\begin{equation}\label{DWcishu}
	\mathcal{K}\left( N \right) = {N^{\frac{1}{{d + 2}}}}{\left( {\ln N} \right)^{ - \frac{\zeta }{{d + 2}}}} ,
\end{equation}
we have 
\[\mathcal{K}\left( N \right) = {\mathcal{O}^\# }\left( {2e^{-1}\sqrt {\frac{{2\pi \alpha N}}{{\tilde \lambda \left( {1,1} \right)\mathcal{K}{{\left( N \right)}^{d }}{{\ln }^\zeta }\left( {1 + \mathcal{K}\left( N \right)} \right)}}} } \right),\]
which leads to 
\[	\left|{\mathscr{S}_{\rm T}}\right|+\left| {{\mathscr{S}_{\rm R}}} \right|=\mathcal{O}\left( {\exp \left( { - {c_{\rm I}}{N^{\frac{1}{{d + 2}}}}{{\left( {\ln N} \right)}^{ - \frac{\zeta}{{d + 2}}}}} \right)} \right)\]
with some universal constant $ {c_{\rm I}}>0 $, and so does the  Birkhoff error term $ {\bf Error}_{N} $ by recalling  \eqref{DWFENDUAN}. This proves (I).

The proofs of (II) and (III) only require minor modifications to the proof of (I). It is noted that the nonresonant condition in (II) corresponds to the case where $ \zeta = 0  $ in \eqref{DWALMOST}, and indeed, it does not affect the subsequent analysis in (I) as $ f $ is still analytic. Therefore, the Birkhoff error term $ {\bf Error}_{N} $  in (II)  admits an upper bound as $ \mathcal{O}( {\exp ( { - {c_{\rm II}}{N^{\frac{1}{{d + 2}}}}} )} ) $,
where $ c_{\rm II} >0$ is some universal constant. Regarding  (III), which specifically considers that $ f $ is merely  a finite trigonometric polynomial and is applied for  almost all rotations,  the truncation technique in (I) is unnecessary.  In other words, $ \mathcal{K}(N)=\mathcal{O}(1) $, $ m_N = \mathcal{O}^\#(\sqrt N) $, and the reminder term $ \mathscr{S}_{\rm R} $ does not exist. Consequently, the estimate for the truncated term $ \mathscr{S}_{\rm T} $ in \eqref{DWTRUNC} directly yields $ {\bf Error}_{N}=\mathcal{O}( {\exp ( { - {c_{\rm III}}\sqrt N } )} ) $ in (III) with some universal constant $ c_{\rm III} >0$.  
\end{proof}

Building directly on Theorem \ref{DWJQ}, we employ the crucial Lemma \ref{DWLEMMA51} (detailed in Section \ref{DWPQZHISHU}) on the weighting function $ w_{p,q}(x) $ to establish the following result of so-called nearly linear-exponential convergence.  We prefer to present this finding in this section, rather than as a primary theorem, due to our focus on the specific exponential function $ w_{1,1}(x) $, which was initially utilized by Laskar. We would also like to refer to Theorem III.2 by Ruth and Bindel in \cite{MR4833591}, although it is quite different from our case.

\begin{theorem}\label{DWRYZS}
		Consider the weighted Birkhoff error term $ {\bf Error}_{N} $ in \eqref{DWERROR}, where the weighting function is replaced by $ w_{p,q}(x) $ for $ p,q\geqslant 1 $. Assume further that the observable is ``super-analytic'', defined by its Fourier coefficients satisfying $ |f_k| =\mathcal{O}\left(\exp\left(-|k|^{v}\right)\right) $ for any $ v>1 $. Then, for any arbitrarily given $ 0<\mathbf{a} <1$ and almost all rotations $ \rho \in \mathbb{T}^d $, there exist sufficiently large $ p,q\geqslant 1 $ such that
		\[{\mathbf{Error}}_N = \mathcal{O}\left( {\exp \left( { - {c_{\rm IV}}{N^{\mathbf{a}}}} \right)} \right),\]
		where $ c_{\rm IV} >0$ is some universal constant. 
\end{theorem}
\begin{remark}
Theorem \ref{DWRYZS} serves as a supplement of the two fundamental questions (Q1) and (Q2) discussed in  Section \ref{DWPQZHISHU} and also the work of \cite{MR4768308,TLCesaro,TLmultiple,MR3718733,MR4665059}.
\end{remark}
\begin{proof}
	The analysis follows a similar approach to that in Theorem \ref{DWJQ}.  Denote $ {\mathbf{b}} := 1 + \frac{1}{{\min \left\{ {p,q} \right\}}}\in (0,1) $ and let $ v>1 $ be arbitrarily given. Here, we utilize the nonresonant condition  \eqref{DWALMOST}, which holds for almost all rotations $ \rho \in \mathbb{T}^d $. Define the truncation order as  $ \mathcal{K}\left( N \right) = {N^{\frac{1}{{d + v{\mathbf{b}}}}}}{\left( {\ln N} \right)^{ - \frac{\zeta }{{d + v{\mathbf{b}}}}}} $. By applying the truncation technique and Lemma \ref{DWLEMMA51}, we obtain estimates analogous to \eqref{DWTRUNC} and \eqref{DWIIIIIII}:
\[\left| {{\mathscr{S}_{\rm{T}}}} \right| = \mathcal{O}\left( {\exp \left( { - 2{e^{ - 1}}{{\left( {\frac{{2\pi \alpha N}}{{\tilde \lambda \left( {p,q} \right)\mathcal{K}{{\left( N \right)}^d}{{\ln }^\zeta }\left( {1 + \mathcal{K}\left( N \right)} \right)}}} \right)}^{1/{\mathbf{b}}}}} \right)} \right),\]
and
\[	\left| {{\mathscr{S}_{\rm R}}} \right| \leqslant \sum\limits_{0 \ne k \in {\mathbb{Z}^d},\left| k \right| \geqslant \mathcal{K}\left( N \right)} {|{{\hat f}_k}|\frac{1}{{{A_N}}}\sum\limits_{n = 0}^{N - 1} {w\left( {n/N} \right)} }  = \sum\limits_{0 \ne k \in {\mathbb{Z}^d},\left| k \right|\geqslant \mathcal{K}\left( N \right)} {|{{\hat f}_k}|}  = \mathcal{O}\left( {{\exp\left(- {\mathcal{K}}\left( N \right)^v\right)}} \right),\]
where the latter follows from the ``super-analyticity'' of $ f $. Given the choice of $ \mathcal{K}\left( N \right)  $, we arrive at 
\[{\mathbf{Error}}_N \leqslant\left| {{\mathscr{S}_{\rm{T}}}} \right| + \left| {{\mathscr{S}_{\rm{R}}}} \right| = \mathcal{O}\left( {\exp \left( { - {c_{{\rm IV}}}{N^{\frac{v}{{d + v{\mathbf{b}}}}}}{{\left( {\ln N} \right)}^{ - \frac{{\zeta v}}{{d + v{\mathbf{b}}}}}}} \right)} \right)\]
for some universal $ {c_{{\rm IV}}}>0 $. This establishes the desired conclusion for sufficiently large  $ v,p,q>1 $, since $0< \frac{v}{{d + v{\mathbf{b}}}} <1 $, and 
\[\mathop {\lim }\limits_{p,q \to  + \infty } \mathop {\lim }\limits_{v \to  + \infty } \frac{v}{{d + v{\mathbf{b}}}} = \mathop {\lim }\limits_{p,q \to  + \infty } {{\mathbf{b}}^{ - 1}} = \mathop {\lim }\limits_{p,q \to  + \infty } {\left( {1 + \frac{1}{{\min \left\{ {p,q} \right\}}}} \right)^{ - 1}} = 1.\]
\end{proof}

We end this section by mentioning that the above analysis can be adapted to address the almost periodic case (though more complex), as discussed in \cite{TLmultiple}.  This is achieved by utilizing the infinite-dimensional spatial structure provided by Montalto and Procesi \cite{MR4201442}, among others. We plan to delve into this topic in future research.

\section{Numerical simulation and analysis of convergence rates}\label{DWSEC7}
In this section, we present  an example to illustrate our quantitative estimates of the convergence rates as stated in Theorem \ref{SWT1}. Let the decaying parameter be $ \lambda=2 $,  the rotating parameter be $ \rho=3 $,  and the initial phase parameter be  $ \theta=1 $. In this case,
\[{{\rm{DW}}_N}\left( {2 ,3 ,1 } \right) = \frac{1}{N}\sum\limits_{n = 0}^{N - 1} {{e^{ - 2 n}}\sin \left( {1  + 3n } \right)},\;\;{\rm{WDW}}_N\left( {2 ,3 ,1 } \right) = \frac{1}{{{A_N}}}\sum\limits_{n = 0}^{N - 1} {w\left( {n/N} \right){e^{ - 2 n}}\sin \left( {1  + 3n } \right)}. \]
It can be verified that $ {e^{ - \lambda }}\sin \left( {\rho  - \theta } \right) + \sin \theta  = {e^{ - 2}}\sin 2 + \sin 1 \approx 0.96 \ne 0 $, and
 \[\xi  =\xi (\lambda,\rho)= \sqrt {2\lambda }  + {e^{ - 1}}\sqrt {{\lambda ^2} + \left\| \rho  \right\|_{\mathbb{T}}^2}  = 2 + \sqrt {13} {e^{ - 1}}.\]
  Therefore, Theorem \ref{SWT1} tells us  that the unweighted average $ {{\rm{DW}}_N}\left( {2 ,3 ,1 } \right) $ exhibits polynomial convergence of $ \mathcal{O}^{\#}\left(N^{-1}\right) $, while the weighted average ${\rm{WDW}}_N\left( {2 ,3 ,1 } \right)  $ demonstrates exponential convergence as given by
\begin{equation}\label{DWLIZI}
	\mathcal{O}\left( {\exp \left( { - \xi \sqrt N } \right)} \right) = \mathcal{O}\left( {\exp \left( { - \left( {2 + \sqrt {13} {e^{ - 1}}} \right)\sqrt N } \right)} \right).
\end{equation}
  We would like to emphasize that  \eqref{DWLIZI} is exceptionally  precise. For $ N=100 $, \eqref{DWLIZI} provides an upper bound {{(without the universal control coefficient)}} on the absolute error of approximately $ 3.58 \times {10^{ - 15}} $, while the actual error is around $ 2.04 \times {10^{ - 15}} $, as illustrated  in Figure \ref{FIGDW1}; for $ N=1000 $, \eqref{DWLIZI} yields  an upper bound {{(without the universal control coefficient)}} on the absolute error of approximately $ 2.07 \times {10^{ - 46}} $, compared to the actual error of approximately $ 2.11 \times {10^{ - 47}} $, {{see details in Figure \ref{FIGDW2}. There, we set the variable on the $ x $-axis to $ \sqrt{N} $. Since the variable on the $ y $-axis changes logarithmically, the theoretical error (indicated in orange) is a straight line with a slope of $ -\xi=-(2+\sqrt{13}e^{-1}) $. This closely matches the variation in the actual error (indicated in purple), thereby demonstrating the accuracy of our estimates in this paper.}}

\begin{figure}[h] 	\centering 	\includegraphics[width=300pt]{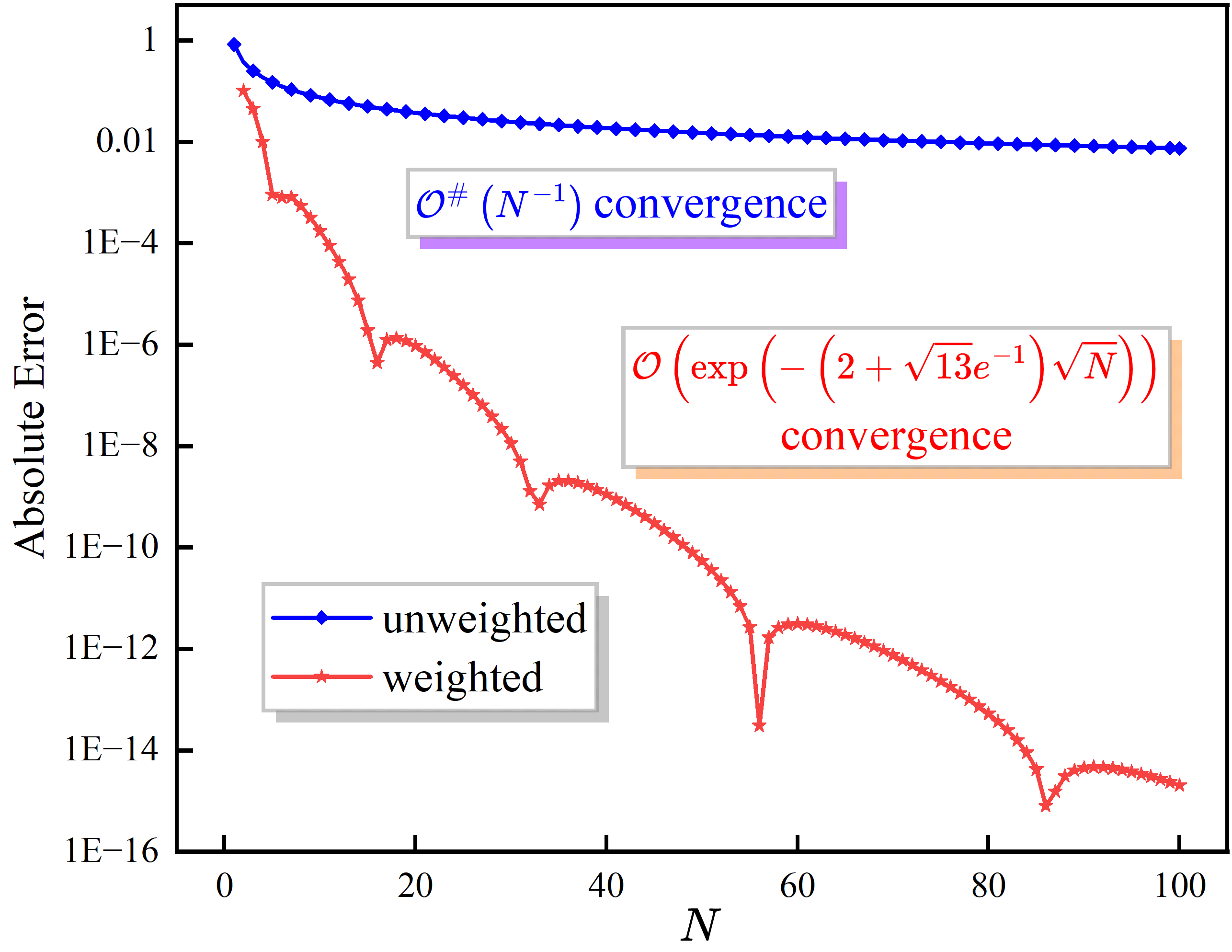} 	\caption {Comparison  of the  convergence rates} 	\label{FIGDW1} \end{figure}

\begin{figure}[h] 	\centering 	\includegraphics[width=300pt]{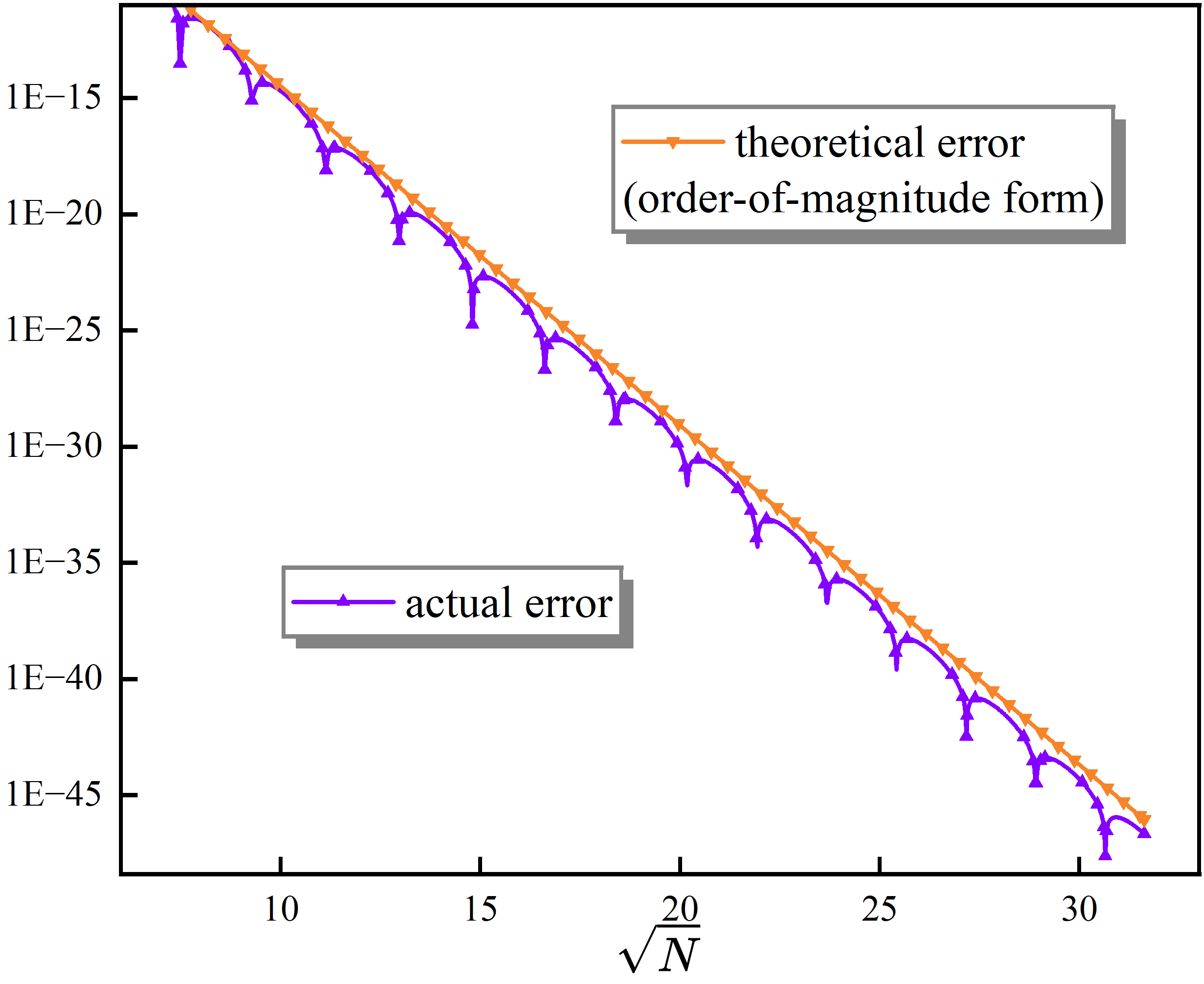} 	\caption {Comparison of the theoretical error with the actual error} 	\label{FIGDW2} \end{figure}

\setcounter{footnote}{0}
\renewcommand{\thefootnote}{\fnsymbol{footnote}}
Actually, on the standard torus $  \mathbb{T} = \mathbb{R}/2\pi \mathbb{Z} $, the rotation number with respect to the rotating parameter in this case is $ 3/2\pi  $, which admits a Diophantine exponent of $ \delta =6.11$\footnote{Indeed, for any irrational $ x $ with an exact Diophantine exponent, $ \frac{{ax + b}}{{cx + d}} $ admits the same exact Diophantine exponent as $ x $, provided that $ a,b,c,d $ are rational and $ ad-bc \ne 0 $. As a consequence, $ \frac{3}{{2\pi }} $ admits the same exact Diophantine exponent as  $ \pi$, as seen in the equivalent expression $ \frac{3}{{2\pi }} + 1 = \frac{{2\pi  + 3}}{{2\pi  + 0}} $.}, 
i.e.,
\[\left| {q \cdot \frac{3}{{2\pi }} - p} \right| > \frac{C}{{{{\left| q \right|}^\delta }}},\;\;\forall q \in {\mathbb{N}^ + },\;\;\forall p \in \mathbb{Z}\]
for some $ C>0 $, see Zeilberger and Zudilin  \cite{MR4170705} for a more accurate  estimate on this aspect. We intentionally avoided using the \textit{constant type} (with exact Diophantine exponent $ 2 $\footnote{It is well known that such rotation numbers only form a set of zero Lebesgue measure in $ \mathbb{R} $.}) rotation numbers that would have further accelerated the convergence rate. Instead, we select a less irrational (more universal) alternative as $ 3/2\pi  $. We also refer to Das et al. \cite{MR3718733} for a numerical comparison of the two cases in the \textit{weighted Birkhoff average} with rotation numbers $ \pi-3 $ (having the same Diophantine exponent as $ 3/2\pi  $) and $ \sqrt{2} -1$ (the constant type).

 \section*{Acknowledgements} 
 This work is supported in part by the National Natural Science Foundation of China (Grant Nos. 12071175 and 12471183). The first author extends heartfelt gratitude to Professors Wolfgang M. Schmidt, Paul Vojta, Doron Zeilberger, and Wadim Zudilin for their insightful email correspondence on Diophantine approximation. Additionally, the first author acknowledges Professors Maximilian Ruth and David Bindel for their enlightening email exchanges regarding the precise convergence rates in weighted Birkhoff averages.  The first author also wishes to express appreciation to Professors Aleksandr G. Kachurovski\u{\i}, Ivan V. Podvigin, and Valery V. Ryzhikov for their invaluable assistance and discussions. {{Finally, the authors would also like to sincerely thank the anonymous referees for their valuable suggestions and comments which led to a significant improvement of the paper.}}

\end{document}